\documentclass[a4paper]{article}
\usepackage{amsthm,amssymb,amsmath,enumerate,graphicx,epsf}

\newcommand{\COLORON}{0}
\newcommand{\NOTESON}{0}
\newcommand{\Debug}{0}
\usepackage[usenames,dvips]{color} 
\usepackage{amsthm,amssymb,amsmath,enumerate,graphicx,epsf,stmaryrd}

\newcommand{\comment}[1]{}
\newcommand{\COMMENT}[1]{}

\definecolor{darkgray}{rgb}{0.3,0.3,0.3}
\newcommand{\defi}[1]{{\color{darkgray}\emph{#1}}}

\newcommand{\acknowledgement}{\section*{Acknowledgement}}

\comment{
	\begin{lemma}\label{}	
\end{lemma}
\begin{proof}

\end{proof}

\begin{theorem}\label{}
\end{theorem} 
\begin{proof} 	

\end{proof}

}

\newtheorem{proposition}{Proposition}[section]
\newtheorem{definition}[proposition]{Definition}
\newtheorem{theorem}[proposition]{Theorem}
\newtheorem{corollary}[proposition]{Corollary}

\newtheorem{lemma}[proposition]{Lemma}
\newtheorem{observation}[proposition]{Observation}
\newtheorem{conjecture}{{Conjecture}}[section]

\newtheorem{problem}[conjecture]{{Problem}}

\newtheorem{examp}[proposition]{Example}

\newcommand{\FIG}{0}

\ifnum \NOTESON = 1 \newcommand{\note}[1]{ 

\hspace*{-30pt}
	{\color{blue}  NOTE: \color{Turquoise}{\small  \tt \begin{minipage}[c]{1.1\textwidth}  #1 \end{minipage} \ignorespacesafterend }} 
	
	}
\else \newcommand{\note}[1]{} \fi

\newcommand{\afsubm}[1]{ \ifnum \Debug = 1 {\mymargin{#1}}
\fi} 

\ifnum \Debug = 1 \newcommand{\sss}{\ensuremath{\color{red} \bowtie \bowtie \bowtie\ }}
\else \newcommand{\sss}{} \fi

\ifnum \FIG = 1 \newcommand{\fig}[1]{Figure ``{#1}''}
\else \newcommand{\fig}[1]{Figure~\ref{#1}} \fi

\ifnum \FIG = 1 
\else  \fi

\ifnum \Debug = 1 \usepackage[notref,notcite]{showkeys}
\fi

\ifnum \COLORON = 0 \renewcommand{\color}[1]{}
\fi

\newcommand{\showFig}[2]{
   \begin{figure}[htbp]
   \centering
   \noindent
   \epsfbox{#1.eps}
   \caption{\small #2}
   \label{#1}
   \end{figure}
}

\newcommand{\N}{\ensuremath{\mathbb N}}
\newcommand{\R}{\ensuremath{\mathbb R}}
\newcommand{\Z}{\ensuremath{\mathbb Z}}

\newcommand{\ca}{\ensuremath{\mathcal A}}

\newcommand{\cc}{\ensuremath{\mathcal C}}

\newcommand{\ci}{\ensuremath{\mathcal I}}

\newcommand{\cn}{\ensuremath{\mathcal N}}
\newcommand{\co}{\ensuremath{\mathcal O}}
\newcommand{\cp}{\ensuremath{\mathcal P}}

\newcommand{\cs}{\ensuremath{\mathcal S}}

\newcommand{\cx}{\ensuremath{\mathcal X}}

\newcommand{\cw}{\ensuremath{\mathcal W}}

\newcommand{\gam}{\ensuremath{\gamma}}
\newcommand{\Gam}{\ensuremath{\Gamma}}
\newcommand{\del}{\ensuremath{\delta}}
\newcommand{\eps}{\ensuremath{\epsilon}}

\newcommand{\lam}{\ensuremath{\lambda}}
\newcommand{\sig}{\ensuremath{\sigma}}

\newcommand{\ccc}{\ensuremath{\mathcal C}}

\newcommand{\sm}{\backslash}
\newcommand{\restr}{\upharpoonright}
\newcommand{\sydi}{\triangle}

\newcommand{\cls}[1]{\overline{#1}}

\newcommand{\iin}{\ensuremath{{i\in\N}}}
\newcommand{\unin}{\ensuremath{[0,1]}}

\newcommand{\sgl}[1]{\ensuremath{\{#1\}}}
\newcommand{\pth}[2]{\ensuremath{#1}\text{--}\ensuremath{#2}~path}

\newcommand{\pths}[2]{\ensuremath{#1}\text{--}\ensuremath{#2}~paths}

\newcommand{\seq}[1]{\ensuremath{(#1_n)_{n\in\N}}} 
 
\newcommand{\g}{\ensuremath{G\ }}
\newcommand{\G}{\ensuremath{G}}

\newcommand{\knl}{Kirchhoff's node law}
\newcommand{\kcl}{Kirchhoff's cycle law}

\newcommand{\are}{\vec{e}}
\newcommand{\arE}{\vec{E}}

\newcommand{\Lr}[1]{Lemma~\ref{#1}}
\newcommand{\Lrs}[1]{Lemmas~\ref{#1}}
\newcommand{\Tr}[1]{Theorem~\ref{#1}}

\newcommand{\Sr}[1]{Section~\ref{#1}}
\newcommand{\Srs}[1]{Sections~\ref{#1}}

\newcommand{\Cr}[1]{Corollary~\ref{#1}}
\newcommand{\Cnr}[1]{Con\-jecture~\ref{#1}}
\newcommand{\Or}[1]{Observation~\ref{#1}}

\newcommand{\Dr}[1]{De\-fi\-nition~\ref{#1}}

\newcommand{\lfg}{locally finite graph}

\newcommand{\Cg}{Cayley graph}

\renewcommand{\iff}{if and only if}
\newcommand{\fe}{for every}
\newcommand{\Fe}{For every}

\newcommand{\st}{such that}

\newcommand{\ti}{there is}
\newcommand{\ta}{there are}

\newcommand{\obda}{without loss of generality}

\newcommand{\wrt}{with respect to}

\newcommand{\leth}{large enough that}

\newcommand{\rw}{random walk}

\newcommand{\labequ}[2]{ \begin{equation} \label{#1} #2 \end{equation} } 
\newcommand{\labtequ}[2]{
 \begin{equation} \label{#1} 	\begin{minipage}[c]{0.9\textwidth}  #2 \end{minipage} \ignorespacesafterend \end{equation} } 

\newcommand{\labtequc}[2]{ \begin{equation} \label{#1} 	\text{   #2 } \end{equation} } 

\newcommand{\labtequstar}[1]{ \begin{equation*}  	\begin{minipage}[c]{0.9\textwidth}  #1 \end{minipage} \ignorespacesafterend \end{equation*} }

\newcommand{\mymargin}[1]{
  \marginpar{%
    \begin{minipage}{\marginparwidth}\small%
      \begin{flushleft}%
        {\color{blue}#1}%
      \end{flushleft}%
   \end{minipage}%
  }%
}%

\newcommand{\mySection}[2]{}

\newcommand{\stoco}{stochastically convergent}
\newcommand{\prj}[1]{\ensuremath{\overline{#1}}}

\newcommand{\squt}{square tiling}
\newcommand{\ded}[1]{\ensuremath{\overrightarrow{#1}}}

\newcommand{\PBv}{Poisson boundary}
\newcommand{\Lzol}{Levy's zero-one law}
\newcommand{\BCl}{Borel-Cantelli lemma}
\newcommand{\shf}{sharp harmonic function}

\newcommand{\ExI}[1]{\ensuremath{\mathbb E\left[#1\right]}}
\newcommand{\PrI}[1]{\ensuremath{\mathbb{P}\left[#1\right]}}

\newcommand{\PrII}[2]{\ensuremath{\mathbb P_{#1}\left[#2\right]}}

\newcommand{\Puz}{\ensuremath{P^c_\sswarrow}}

\newcommand{\hybg}{\ensuremath{\partial_h(G)}}
\newcommand{\cbg}{\ensuremath{\partial_\sim (G)}}
\newcommand{\cgbg}{\ensuremath{\partial_{\cong} (G)}}
\newcommand{\vapf}{accumulation-free}

\newcommand{\clos}[1]{\overline{#1}}
\newcommand{\as}{almost surely}

\title{The Boundary of a Square Tiling of a Graph coincides with the Poisson Boundary}

\author{Agelos Georgakopoulos\thanks{Partly supported by FWF Grant P-24028-N18.}
\medskip 
\\
  {Mathematics Institute}\\
 {University of Warwick}\\
  {CV4 7AL, UK}\\}

\begin{document}
\maketitle

\begin{abstract}
Answering a question of Benjamini \& Schramm \cite{BeSchrST}, we show that the Poisson boundary of any planar, uniquely absorbing (e.g.\ one-ended and transient) graph with bounded degrees can be realised geometrically as a circle, namely as the boundary of a tiling of a cylinder by squares. This implies a conjecture of Northshield \cite{NorCir} of similar flavour. For our proof we introduce a general criterion for identifying the Poisson boundary of a stochastic process that might have further applications.
\end{abstract}

\section{Introduction}

\subsection{Overview} \label{over}

In this paper we prove the following fact, conjectured by Benjamini \& Schramm \cite[Question~7.4.]{BeSchrST} 

\begin{theorem} \label{main}
If \g is a plane, uniquely absorbing\footnote{See \Sr{prelim} for definitions.} graph with bounded degrees, then the boundary of its square tiling is a realisation of its Poisson boundary.
\end{theorem}

The main implication of this is that every such graph \g can be embedded inside the unit disc ${\mathbb D}$ of the real plane in such a way that \rw\ on \g converges to $\partial {\mathbb D}$ almost surely, its exit distribution coincides with Lebesgue measure on $\partial {\mathbb D}$, and \ti\ a one-to-one correspondence between the bounded harmonic functions on \g and $L^\infty(\partial {\mathbb D})$ (the innovation of \Tr{main} is the last sentence).

Knowing that a graph \g is not uniquely absorbing provides a lot of information about its Poisson boundary; in particular, $G$ is not Liouville. Thus it is not a significant restriction in our context to assume that \g is uniquely absorbing.

\medskip
We use \Tr{main} to prove a conjecture of Northshield \cite{NorCir} of similar flavour (\Cnr{Ncon}), thus obtaining a further geometric realisation of the Poisson boundary of the graphs in question.

In the case where \g is hyperbolic, we can say a bit more:

\begin{corollary}\label{cor}
Let \g be an infinite, Gromov-hyperbolic,  non-amenable, 1-ended, plane graph with bounded degrees and no infinite faces. Then the following five boundaries of \G\ (and the corresponding compactifications of \G) are canonically homeomorphic to each other: the hyperbolic boundary, the Martin boundary, the boundary of the square tiling, the Northshield circle, and the boundary \cgbg. 
\end{corollary}
 
We present examples showing that all these conditions are necessary to make the statement true and none of them is implied by the others, except that it is not clear whether both the bounded degree and the non-amenability conditions are necessary; see Problems~\ref{Pram} and \ref{Prhyp}. The least obvious case is to show that hyperbolicity is not implied by the other properties, contrary to another conjecture of Northshield \cite{NorCir}; we disprove that conjecture by a counterexample in \Sr{secHyp}.

The equivalence of the Martin boundary to the hyperbolic boundary in \Cr{cor} follows from a result of Ancona \cite{AncNeg,Anc} that can be thought of as a discrete version of the theorem of Anderson \& Schoen \cite{AnSchPos} that for any complete, simply connected Riemannian manifold with bounded and negative sectional curvatures, the Martin boundary with respect to the Laplace-Beltrami operator coincides with the geometric boundary. This motivates

\begin{conjecture}
Let $M$ be a complete, simply connected Riemannian surface with Gaussian curvatures bounded between two negative constants. Let $f: M \to \mathbb{D}$ be a conformal map from $M$ to the open unit disc in $\mathbb{C}$. Then for every  1-way infinite geodesic \gam\ in $M$, the image $f(\gam)$ converges to a point in the boundary $\mathbb{S}^1$ of $\mathbb{D}$, and this convergence determines a homeomorphism from the sphere at infinity of $M$ to $\mathbb{S}^1$. (In particular, images of equivalent geodesics converge to the same point of $\mathbb{S}^1$.)
\end{conjecture}

\medskip
For the proof of \Tr{main} the following criterion for identifying the Poisson boundary of a general Markov chain is introduced, that might have further applications. A function $h: V\to [0,1]$ is \defi{sharp}, if its values along the trajectory of the Markov chain converge to 0 or 1 almost surely.

\begin{theorem}\label{thPB}
Let $M$ be an irreducible Markov chain and $\cn$ an $M$-boundary. Then \cn\ is a realisation of the Poisson boundary of $G$ \iff\ it is faithful to every sharp harmonic function.

\end{theorem} 

Loosely speaking, \Tr{thPB} states that in order to check that a candidate space  is a realisation of the Poisson boundary, it suffices to consider its behaviour \wrt\ the \shf s rather than all bounded harmonic functions. This fact, and its proof, can be generalised to many continuous stochastic processes.
We will have a closer look at \Tr{thPB} and its implications in \Sr{our}, after some background information.

\subsection{Square tilings ---a discrete analogue of conformal uniformization}

Motivated by Dehn's problem of dissecting a rectangle into squares of distinct side lengths \cite{Dehn}, Brooks et.\ al.\ \cite{BSST}  showed how any finite planar graph \g can be associated with a tiling of a rectangle by squares in such a way that every edge $e$ of \g corresponds to a square $R_e$ in the tiling and every vertex $x$ corresponds to an interval tangent with all squares corresponding to the edges of $x$. The construction was made using an electrical current on the graph, and the square $R_e$ is given side length equal to the current going through $e$, while the position of $R_e$ is determined by the voltages of the vertices (see \Sr{constr} for details).  This construction has become a classic, with a lot of applications 
ranging from recreational mathematics to statistical mechanics; see e.g.\ \cite{reseaux,DujSim,kenyon_tilings_1998,kenyon_dimers_2004}.

Benjamini and Schramm \cite{BeSchrST} showed that the construction of \cite{BSST} can be applied to an infinite planar graph \g in the uniquely absorbing case; the electrical current now emanates from a single vertex $o$ and escapes to infinity, and is intimately connected to the behaviour of \rw\ from $o$. The square tiling takes place on the cylinder $K=\R/\Z \times [0,1]$, the edges of the graph being mapped to disjoint squares tiling $K$; see \fig{treest} or \cite[FIG. 1]{BeSchrST} for examples.
Their motivation was to find discrete analogues of Riemann's mapping theorem; the following quote is from  \cite{BeSchrST}

{\quote \small ``The tiling plays the same role as conformal uniformization does for planar domains. In fact, the proof of its existence illustrates parallels with the continuous theory. In a way, it is a discrete analogue of Riemann's mapping theorem.''}
\medskip

For more on the relationship between square tilings and Riemann's mapping theorem see \cite{CaFlPaSqu}. The paper \cite{BeSchrHar} is similar in nature to \cite{BeSchrST} and provides a further discrete analogue of Riemann's mapping theorem. A well known corollary of both \cite{BeSchrHar,BeSchrST} is that every planar transient graph admits non-constant harmonic functions of finite energy. This result, in fact a detailed description of the space of such functions, can now be derived combining \Tr{main} with the results of \cite{theta}.

\showFig{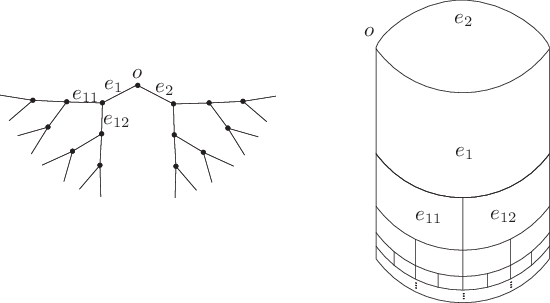}{The infinite binary tree and its square tiling.}

\medskip
The quest for discrete occurrences of conformal invariance, either as discretised analogues of the continuous theory, or as original constructions, is imminent in much of the work of Schramm and was followed up by many others, leading to several impressive applications; see \cite{RohOded} for a survey.

\subsection{The Poisson boundary of a stochastic process}

The Poisson(-Furstenberg) boundary of a transient stochastic process $M$ is a measure space $\cp$ associated with $M$ \st\ every bounded harmonic function on the state space of $M$ can be represented by an integral on $\cp$ and, conversely, every function in $L^\infty(\cp)$ can be integrated into a bounded harmonic function of $M$. The standard example \cite{FurPoi,KaiPoHypAnn} is when $M$ denotes Brownian motion on the open unit disc $\mathbb D$ in the complex plane; then the classical Poisson integral representation formula $h(z) = \int_0^1 \frac{1-|z|^2}{|e^{2\pi i\theta} - z|^2}\hat{h}(\theta)d\theta =   \int_0^1 \hat{h}(\theta) d\nu_z(\theta)$ recovers every continuous harmonic function $h$ from its boundary values $\hat{h}$ on the circle $\partial \mathbb D$. Here, $\nu_z$ is the \defi{harmonic measure} on $\partial \mathbb D$, i.e.\ the exit distribution of Brownian motion started at $z$, and can be obtained by multiplying the Lebesgue measure by the Poisson kernel $\frac{1-|z|^2}{|e^{2\pi i\theta} - z|^2}$. Thus, we can identify the \PBv\ of Brownian motion $M$ on $\mathbb D$ with $\partial \mathbb D$, endowed with the family of measures $\nu_z, z\in \mathbb D$. 

In this example, the boundary was geometric and obvious. If $M$ is now an arbitrary transient Markov chain, then \ti\ an abstract construction of a measurable space $\cp$, called the \PBv\ of $M$, endowed with a family of measures $\nu_z$ indexed by the state space $V$ of $M$, such that the formula $h(z) = \int_\cp \hat{h} d\nu_z$ provides an isometry between the Banach space $H^\infty(M)$ of bounded harmonic functions of $M$ (endowed with the supremum norm) and the space $L^\infty(\cp)$.

Triggered by the work of Furstenberg \cite{FurPoi,FurRan}, a lot of research has concentrated on identifying the \PBv\ of various Markov chains, most prominently locally compact groups endowed with some measure; see \cite{AncNeg,kaimanovich_poisson_1996,KaVeRan,karlsson_poisson_2006} just to mention some examples, and \cite{erschlerICM} for a survey including many references and some impressive applications. However, I would like to stress that this paper is not about groups (although some of its techniques might be applicable to them), but rather about extending the study of the \PBv\ to embrace general graphs, a currently emerging quest \cite{AnBaGuNa,theta}.

In general, one would like to identify the \PBv\ of a given Markov chain with a geometric object as we did for $\mathbb D$, which could for example be a compactification of a \Cg\ on which our group acts. This task can however be very hard. Some general criteria have been developed for the case of groups, mostly by Kaimanovich \cite{KaiPoHypAnn,KaVeRan}, that have helped to identify `geometric' Poisson boundaries for certain classes of groups 
 e.g.\ hyperbolic groups, but others, like the lamplighter group over $Z^3$ \cite{erschler_poissonfurstenberg_2011}, defy such identification despite extensive efforts. A well-known open problem is whether the Liouville property, which can be expressed as triviality of the \PBv,  for simple random walk on a \Cg\ $G= Cay(\Gamma,S)$ is a group invariant, i.e.\ independent of the choice of the generating set $S$ of a given group $\Gamma$.

The aforementioned general criteria for geometric identification of the \PBv\ only apply to groups. \Tr{thPB}, and its corollary \Cr{cortop} below, may be the first general criterion for arbitrary Markov chains, while still being applicable in the case of groups.

\subsection{More on our results} \label{our}
There are various definitions of the Poisson boundary for general Markov chains in the literature. In this paper, rather than choosing one of them, or introducing our own, we will follow a more flexible approach, accepting any measure space that fulfils the properties expected from ``the'' \PBv\ as ``a'' \PBv. More precisely, given a transient Markov process $M$, we let an \defi{$M$-boundary}, or a \defi{$G$-boundary} if $M$ happens to be \rw\ on a graph \G, be any measurable space $\cn$ endowed with a family of measures $(\nu_z), z\in V$ (recall our $\mathbb D$ example) and a measurable, measure preserving, shift-invariant function $f$ from the space of \rw\ trajectories $\cw$ to $\cn$; see \Sr{secSh} for details. This definition is rather standard \cite{KaiPoHypAnn}. We say that an $M$-boundary $\cn$ is a \defi{realisation} of the Poisson boundary of $M$, if every bounded harmonic function $h$ on the state space $V$ of $M$ can be obtained by integration of a bounded function $\hat{h}\in L^\infty(\cn)$, where this $\hat{h}$ is unique up to modification on a null-set, and conversely, \fe\ $\hat{h}\in L^\infty(\cn)$ the function $f:V\to \R$ defined by $z\mapsto \int_\cn \hat{h}(\eta)d\nu_z(\eta)$ is bounded and harmonic.

Recall the definition of a sharp function from \Sr{over}. In \Sr{secSh} we introduce `intersection' and `union' operations between pairs or families of sharp harmonic functions using probabilistic intuition, and show that the family $\cs$ of \shf s of a Markov chain is closed under these operations. Thus $\cs$ carries a \sig-algebra structure, except that \ti\ no ground set. \Tr{thPB} can be interpreted in the following way: any $M$-boundary $\cn$ that can be used as the ground set for this \sig-algebra structure, with its measures agreeing with corresponding probabilities defined with respect to \shf s, is a realisation of the Poisson boundary. A bit more precisely, if $\cn$ is \defi{faithful} to $\cs$, that is, if for every measurable subset  $X$ of $\cn$ \ti\ $s\in \cs$ \st\ almost surely \rw\ ends up in $X$ \iff\ the values of $s$ along its trajectory converge to 1, then $\cn$ is  a realisation of the Poisson boundary.

A direct implication of the comparison of the \sig-algebra structure of $\cs$ to that of $\cn$ is that any two realisations of the Poisson boundary of $M$ are cryptomorphic, see \Cr{crypto} (this fact will not be surprising to experts).

Given a Markov chain it is often easy to guess a realisation of its \PBv, most often in the form of some topological space naturally associated to the chain, e.g.\ the boundary of a compactification, but it is much harder to prove that the guess is correct. Our case, the boundary of a square tiling, is such an example. Other examples include the end-compactification of a tree, and the hyperbolic boundary of a hyperbolic group \cite{KaiPoHypAnn}. The following tool, abstracted from the proof of \Tr{main} via \Tr{thPB}, may be helpful in further such cases. A \defi{topological $G$-boundary} of a graph \g is a topological space $(\cn,\co)$ endowed with  a `projection' $\tau: V \to \co$ so that $\tau(Z^n)$ converges to a point in $\cn$ for almost every \rw\ trajectory $Z^n$, and there is a Borel-measurable function $\tau^*: \cw \to \cn$ mapping almost every $(Z^n)\in \cw$ to $\lim_n \tau(Z^n)$.
Note that defining $\nu_z$ by $O\in \co \mapsto \mu_z(\tau^{*-1}(O))$ turns $\cn$ into a $G$-boundary for $f= \tau^*$. We say that $(\cn,\tau)$ is \defi{layered}, if \ti\ a sequence \seq{G} of finite subgraphs of \g with $\bigcup G_n= G$ the boundaries \seq{B}\ of which satisfy $\mu_z^n(b)= \nu_z \circ \tau(b)$ \fe\ $b\in B_n$, where $\mu_z^n$ denotes the exit distribution of $G_n$ for \rw\ from $z$.

\begin{corollary} \label{cortop}
Let $G$ be a transient graph and let $\cn$ be a layered topological $G$-boundary with projection $\tau$. If \fe\ \shf\ $s$ we have $\lim_{m,n} \nu(\tau(F_m) \sydi \tau(F_n)) =0$, where $F_i:= \{b\in B_i \mid s(b)>1/2\}$, then $\cn$ endowed with the measures $(\nu_z)$ as above is a realisation of the Poisson boundary of \rw\ on \G.
\end{corollary}

The following observation, proved in \Sr{secPar}, is one of the main tools in the proof of \Tr{main} and might be of independent interest. Here, a graph is considered as a metric space where every edge is a copy of the real unit interval, and so each square of the tiling is foliated into horizontal intervals, one for each inner point of the corresponding edge.

\begin{observation} \label{obsbm}
Let \g be a plane, uniquely absorbing graph and consider its square tiling $T$ of the cylinder $K$. For any circle $L\subset K$ parallel to the base of $K$, let $B$ be the set of points of \g the images of which lie in $L$. Then the widths $w(T(b)),$ $b\in B$ of these images coincide with the exit probabilities of standard brownian motion on \g started at the reference vertex $o$ and killed at $B$.
\end{observation}

Our proof of \Tr{main} applies to a larger class of graphs where the degrees are not necessarily bounded, see \Cr{unbou}. Such graphs have attracted a lot of interest lately, see \cite{GurNachRec} and references therein.

\medskip
We prove \Tr{thPB} in \Srs{secSh} and \ref{secPr} and \Tr{main} in \Sr{hard}, after constructing the square tiling in \Sr{constr} and proving some general properties, including interesting probabilistic interpretations of its geography like \Or{obsbm}, in \Sr{geo}. We deduce \Cr{cortop} in \Sr{subfaith}.

\section{Preliminaries} \label{prelim}

Let $\g=(V,E)$ be a graph fixed throughout this section, where $V=V(G)$ is its set of vertices and $E= E(G)$ its set of edges. A \defi{walk} on \G\ is a (possibly finite) sequence \seq{v}\ of elements of $V$ \st\ $v_i$ is always connected to $v_{i+1}$ by an edge. More generally, we define a walk on the state space $V$ of a Markov chain in a similar manner, where we might or might not demand that the transition probabilities $p_{v_i \to v_{i+1}}$ be positive.

A \defi{plane graph} is a graph \g endowed with a fixed embedding in the plane $\R^2$; more formally, \g is a plane graph if $V(G)\subset \R^2$ and each edge $e\in E(G)$ is an arc between its two vertices that does not meet any other vertices or edges. A graph is \defi{planar} if it admits an embedding in $\R^2$. Note that a given planar graph can be isomorphic (in the graph-therotic sense) to various plane graphs that cannot necessarily be mapped onto each other via a homeomorphism of $\R^2$.

A plane graph $G\subset \R^2$ is \defi{uniquely absorbing}, if \fe\ finite subgraph $G_0$ \ti\ exactly one connected component $D$ of $\R^2\sm G_0$ that is \defi{absorbing}, that is, \rw\ on \g visits $G \sm D$ only finitely many times with positive probability (in particular, \g is transient).
Uniquely absorbing graphs are precisely those admitting a square tiling. Example classes include all transient 1-ended planar graphs, which includes the bounded-degree 1-ended planar Gromov-hyperbolic graphs, all transient trees and, more generally, all transient graphs that can be embedded in $\R^2$ without accumulation points of vertices. 

We assume that \g is endowed with an assignment $c:E\to \R_+$ of \defi{conductances} to its edges, which are used to determine the behaviour of \rw\ on \g as follows. 

A {\em random walk} on
$G$ begins at some vertex and when at vertex $x$, traverses one of the edges $\ded{xy}$ incident to $x$ according to the probability distribution
\labequ{trpr}{
p_{x\to y}:= \frac{c(xy)}{\pi_x},
}
where $\pi_x:= \sum_{y\in  N(x)} c(xy)$ and $N(x)$ denotes the set of  \defi{neighbours} of $x$, that is, the vertices connected to $x$ by an edge. When $c=1$, which is the case most often considered, $\pi_x$ coincides with the \defi{degree} of $x$, and we have \defi{simple \rw}, i.e.\ $y$ is chosen according to the uniform distribution on $N(x)$.

Formally, there are two standard ways of rigorously formalising \rw\ as a probability space: the first is as a Markov chain in the obvious way. The second, and the one that we will adhere to in this paper, is by considering \rw\ on \g as a measurable space $(\cw, \Pi)$, endowed with a family of measures $(\mu_z)_{z\in V}$ indexed by the vertices of \G, where $\cw$ is the set of 1-way infinite walks on \G, called \defi{path space}, $\Pi$ is the \sig-algebra on $\cw$ generated by the \defi{cylinder sets}, i.e.\ subsets of $\cw$ comprising all walks having a common finite initial subwalk, and $\mu_z$ is the  probability measure on $(\cw, \Pi)$ corresponding to fixing $z$ as the starting vertex. Note that once $z$ is fixed, \eqref{trpr} uniquely determines $\mu_z$; see \cite{WoessBook09} for details.

For convenience, we will also assume that every Markov chain $M$ in this paper is formally given in the above form, that is, a choice of a (possibly random) starting point $o$ and a family of measures $(\mu_z)_{z\in V}$ on path space $(\cw, \Pi)$, where $V$ denotes the state space of $M$ and $\mu_z$ is the law of $M$ conditioning on the starting point being $z$. In other words, we formalize $M$ as a \rw\ on $V$; in contrast to the \rw\ on a graph defined above, such \rw\ need not be reversible, see \cite{LyonsBook}. 
\medskip

It will be convenient in some cases to think of our graph \g as a metric space constructed as follows. Start with the discrete set $V$, and for every edge $xy\in E$ join $x$ to $y$ by an isometric copy of a real interval of length $1/c(e)$ (the \defi{resistance} of $e$). Then, one can consider a brownian motion on this space (as defined e.g.\ in \cite{BC,BPY}) that behaves locally like standard brownian motion on $\R$, and it turns out that the sequence of distinct vertices visited by this brownian motion has the same distribution as \rw\ governed by \eqref{trpr}.

\medskip
A function $h: V \to \R$ on the vertex-set of a graph \G, or more generally on the state space $V$ of a Markov chain, is \defi{harmonic at $x\in V$} if it satisfies
\labtequ{harm}{$h(x) = \sum_{y\in V} p_{x\to y} h(y)$,}
where again $p_{x\to y}$ denotes the transition probability, and it is called \defi{harmonic} if it is harmonic at every $x$.

A fundamental property of harmonic functions is that their values inside a set are determined by the values at the boundary of that set; to make this more precise, let $B$ be a subset of $V$, and $x$ a vertex \st\ \rw\ from $x$ visits $B$ almost surely; for example, $B$ could be the boundary of a ball of \g containing $x$. Then, letting $b$ be the first vertex of $B$ visited by \rw\ from $x$, we have
\labtequ{harmB}{$h(x) = \ExI{h(b)}$.}
In other words, the boundary values of a harmonic function uniquely determine the function.

We say that \rw\ \defi{hits} $B$ at $b\in B$ if the first element of $B$ it visit is $b$.

The following fact, which is a special case of \Lzol\ \cite{karatzas}, will come in handy in many occasions
\labtequ{lzol}{For every $\mu$-measurable event $A\subseteq \cw$, \rw\ $Z^n$ satisfies $\lim_n \mu_{Z^n}(A)={\mathbb 1}_A$ almost surely. In particular, $\mu_{Z^n}(A)$ converges to 0 or 1.}
Here, ${\mathbb 1}_A$ is as usual the characteristic function (from $\cw$ to ${0,1}$) of $A$. In other words, if we observe at each step $i$ of our \rw\ the probability $p_i$ that $A$ will occur given the current position (the past does not matter because of the Markov property), then almost surely $p_i$ will converge to $1$ and $A$ will occur or  $p_i$ will converge to $0$ and $A$ will not occur.

For a walk $W\in \cw$ define the \defi{shift} $t(W)$ to be the walk obtained from $W$ by deleting the first step. An event $A\subseteq \cw$ is called if \defi{tail event} if $t(A)=A$ (we could be less strict here and write $\mu( t(A)\sydi A)=0$ instead, where $\sydi$ denotes symmetric difference).

\section{Sharp functions and the \PBv} \label{secSh}

Let $M= Z^1,Z^2,\ldots$ be an irreducible Markov chain with state space $V$ fixed throughout this section.
Call a harmonic function $h: V\to \R^+$ \defi{sharp}, if the range of $h$ is $[0,1]$ and $\lim_n h(Z^n)$ equals 0 or 1 almost surely (the limit exists almost surely by the bounded martingale convergence theorem). Note that whether $h$ is sharp or not does not depend on the (possibly random) starting point $o$, for if $r$ is any other element of $V$ then, by irreducibility, the probability to visit $r$ from 
$o$ is positive. Let $\cs=\cs(V)$ denote the set of \shf s on $V$.

\begin{lemma} \label{sharph}
If $h(z): V \to [0,1]$ equals the probability that \rw\ from $z$ will satisfy a tail event $A$ \fe\ $z$, then $h$ is a \shf.
\end{lemma}
\begin{proof}
The fact that $h$ is harmonic follows immediately from the fact that $A$ is a tail event and the Markov property. Sharpness follows from \eqref{lzol}.
\end{proof}

Let $s\in \cs$ and fix $z\in V$. Define the real valued random variable $X_n$ to be $s(Z^n)$  where  $Z^n$ denotes \rw\ from $z$. Define the random variable $X$ by letting $X=\lim s(Z^n)$ if this limit exists (which it does almost surely by the bounded martingale convergence theorem), and $X=0$ otherwise. Since almost sure convergence implies weak convergence \cite{Karr}, we immediately obtain
\labtequ{weakly}{The sequence $X_n$ converges weakly to $X$.}
Alternatively, let $\seq{B}$ be a sequence of subsets of $V$ \st\ our Markov chain visits every $B_n$ almost surely, and let $X_n=s(Z^{t_n})$, where $t_n$ is the first time $t$ \st\ $Z^t\in B_n$ (if no such $t$ exists, which happens with probability 0, we can let $X_n=0$ to make sure $X_n$ is always defined). Then $\lim s(X_n)$ still exists almost surely by the choice of $\seq{B}$, and so \eqref{weakly} also holds in this case.

\note{another way to prove this is using Levy's 0-1 law}

As $s$ is harmonic, we have
\labtequ{hzexp}{$s(z)= \ExI{X_n}$.}
Given $z\in V$, let $1^s$ be the event, in the path space $\cw_z$, that for random walk $Z^n$ from $z$ we have $\lim s(Z^n) =1$. Note that this event is $\mu_z$-measurable. Define the event $0^s$ similarly.

\begin{corollary}\label{hprob}	
If $s$ is a \shf, then $s(z)=\mu_z(1^s)$ \fe\ vertex $z$.
\end{corollary}
\begin{proof}
Since the $X_n$ are uniformly bounded, their weak convergence \eqref{weakly} implies convergence in $L^1$, that is, $\lim_n \ExI{X_n}= \ExI{X}$, where $X$ is defined as in \eqref{weakly}. Since $s$ is sharp, $\ExI{X}$ equals the probability to have $\lim_n s(Z^n)=1$ by the definition of $X$. Combined with \eqref{hzexp} and the definition of `sharp' this completes our proof.
\end{proof}

\note{sharp functions correspond to equiv. classes of \stoco\ sequences.}

\begin{corollary}\label{close1}	
If $s$ is a \shf\ that is not identically 0, then \fe\ $\eps>0$ \ti\ $z\in V$ with $s(z)>1-\eps$.
\end{corollary}
\begin{proof}
Since $s\neq 0$, \ti\ $o$ with $s(o)>0$. By \Cr{hprob} the probability to have $\lim_n s(Z^n)=1$ for \rw\ $Z^n$ from $o$ is positive, in particular there are vertices $z$ with $s(z)$ arbitrarily close to 1.
\end{proof}

Given a sequence \seq{s} of sharp harmonic functions, we define their \defi{union} $\bigcup s_i$ by 
$z\mapsto \mu_{z}(\bigcup 1^{s_i})$ and their \defi{intersection} $\bigcap s_i$ by $z\mapsto \mu_{z}(\bigcap 1^{s_i})$. The \defi{complement} $s_1^c$ of $s_1$ is the function $1-s_1$; note that, by \Cr{hprob}, $s_1^c(x) = \mu_x(0^{s_1})$. 

\begin{lemma} \label{comb}
Let $(s_i)$ be a sequence of sharp harmonic functions. Then the functions 
\begin{enumerate}
\item \label{ci} $\bigcup s_i$, 
\item \label{cii} $s_1^c$,
\item \label{ciii} $\bigcap s_i$, 
\end{enumerate}
are also harmonic and sharp.
\end{lemma}
\begin{proof}
Since all these functions are probabilities of tail events of \rw, they are harmonic and  sharp by \Lr{sharph}.
\end{proof}
This means that the family of sharp harmonic functions satisfies the axioms of a \sig-algebra, except that it is formally not a family of subsets of a given set. One way to intuitively interpret \Tr{thPB} is to say that if a measure space \cn\ has the `same' \sig-algebra structure as the family of \shf s of \G, and some obvious requirements are fulfilled, then \cn\ can be identified with the \PBv\ of \G. Let us make this idea more precise.

\medskip
A \defi{$M$-boundary}, or a \defi{$G$-boundary} if $M$ happens to be \rw\ on a graph \G, is a measurable space $(\cn,\Sigma)$ endowed with a family of 
probability measures $\{\nu_z, z\in V \}$ and a measurable, measure preserving, shift-invariant function  $f: \cw \to \cn$, where $\cw:= \bigcup_{z\in V} \cw_z$ is the set of 1-way infinite walks in \G. Here, we say that $f$ is \defi{measure preserving} if $\nu_z(X)=\mu_z(f^{-1}(X))$ \fe\ $ z\in V$ and $X\in \Sigma$\note{(i.e.\ $\nu$ is the push-forward\sss of $\mu$)}. The term \defi{shift-invariant} means that if the walk $W'$ is obtained from $W$ by skipping the first step, then $f(W)=f(W')$; in other words, $f$ can be thought of as a function from the set of equivalence classes \wrt\ the shift (called `ergodic components' in \cite{KaiPoHypAnn}) 
to \cn. 

The fact that $f$ is shift-invariant and measure preserving implies 
\labtequ{harmnu}{$\nu_z = \sum_{y\sim z} p_{zy} \nu_y$.}
Indeed, since $f$ is shift-invariant, \rw\ from $z$ `finishes' in $X\in \Sigma$ \iff\ its subwalk after the first step finishes in $X$, and so we have $ \mu_z(f^{-1}(X)) = \sum_{y\sim z} p_{zy} \mu_y(f^{-1}(X)) $. Since we are demanding that $f$ is measure preserving, our assertion follows.

This means in particular that the measures $\nu_z$ are pairwise equivalent when $G$ is connected: if $\nu_z(X)>0$ for some set $X\in \Sigma$, then $\nu_y(X)>0$ for any neighbour $y$ of $z$.

Let $s$ be a sharp harmonic function. We say that $\cn$ is \defi{faithful} to $s$, if 
\ti\ $X\in \Sigma$ \st\ $\mu_z(1^s \sydi f^{-1}(X)) = 0$ \fe\ $z\in V$, where $\sydi$ denotes symmetric difference. Note that such a set $X$ is unique up to modification by a null-set of $\nu$.

\comment{ 
	\fe\ $z\in V$,
\begin{enumerate}
 \item \label{fi} $f(1^s)$ is $\nu_z$-measurable, and 
 \item \label{fii} $\nu_z(f(1^s))= \mu_{z}(1^s)$.
\end{enumerate}
\note{\ref{fii} does not follow from measure preservingness of $f$}

Note that, by \Cr{hprob}, the latter probability equals $s(z)$.
} 

\medskip
For every measurable subset $X$ of \cn, the function $s=s_X$ defined by $z\mapsto \nu_z(X)$ is harmonic and sharp by \Lr{sharph}. We claim that 
\labtequ{convfai}{$\mu(1^s\sydi f^{-1}(X))=0$}
To see this, set $\Phi := f^{-1}(X)\sydi 1^s$, and recall that both $f^{-1}(X)$ and $1^s$ are $\mu$-measurable events in path space, hence so is $\Phi$. Now note that $\mu_z(f^{-1}(X)\sydi 1^s) =$\\ 
$ \mu_{z}(1^s \cap f^{-1}(X^c))+ \mu_{z}(0^s \cap f^{-1}(X))$,\\ 
where we used the fact that $\mu_{z}(1^s \cup 0^s)=1$ as $s$ is sharp. It is easy to see that both these summands equal 0 using the definition of $s$, the Markov property of \rw, and the fact that $f$ is measure preserving.

Note that if \cn\ is faithful to every \shf, then combined with \eqref{convfai} this means that, up to perturbations by null-sets, the correspondence between measurable subsets of \cn\ and \shf s is one-to-one. Combined with \Tr{thPB}, this observation yields

\begin{corollary} \label{crypto}
If $(\cn,\Sigma)$ is a realisation of the Poisson boundary of a graph $G$, then \ti\ a bijection $\sig$ from the set of equivalence classes of $\Sigma$ (where two elements are equivalent if they differ by a null-set) to the set of \shf s of $G$ \st\ $\nu_z(X)= \mu_z(1^{\sig([X])}_z)$ \fe\ $X\in \Sigma$ and $z\in V(G)$. Thus any two realisations of the Poisson boundary of $G$ are cryptomorphic.
\end{corollary}

\section{Proof of the \PBv\ criterion} \label{secPr}

We start by proving the easier direction of \Tr{thPB}, namely that if $\cn$ is a realisation of the Poisson boundary then it is faithful to every \shf\ $s$. For this, given $s$ let $\hat{s}\in L^\infty(\cn)$ be such that $s(z)= \int_\cn \hat{s}(\eta)d\nu_z(\eta)$ \fe\ $z$. 

We claim that $\hat{s}$ equals 0 or 1 almost everywhere on $\cn$. Indeed, since $f$ is measure preserving, it suffices to prove that $\mu(f^{-1}(X) \cup f^{-1}(Y))=1$, where  $X:= \{\eta \mid \hat{s}(\eta) =1\}$ and $Y:= \{\eta \mid \hat{s}(\eta) =0\}$.
For $\eps\in (0,1)$, let $X_\eps:= \{\eta \mid \hat{s}(\eta)\in [\eps,1-\eps]\}$, and define the function $s_\eps: V\to [0,1]$ by $z\mapsto \mu_z(f^{-1}(X_\eps))$. By \Lr{sharph} $s_\eps$ is harmonic and sharp. Now note that if $s_\eps(W^n)$ converges to 1 for some walk $W^n$, then $s(W^n)$ does not converge to 0 or 1. But as $s$ is sharp, this occurs with probability 0 for our \rw. Since $s_\eps$ is sharp, this implies that $s_\eps(W^n)$ converges to 0 for almost every \rw\ $W^n$. But now \Lzol\ \eqref{lzol} implies that $\mu_z(f^{-1}(X_\eps))=0$ \fe\ $z$. Since this holds \fe\ \eps, our claim follows.

As $\hat{s}$ equals 0 or 1 almost everywhere on $\cn$, we have $s(z) = \nu_z(X)$ by the choice of $\hat{s}$ and $X$. Thus \eqref{convfai} yields $\mu_z(1^s \sydi f^{-1}(X)) = 0$, which means that $\cn$ is faithful to $s$ as desired.

\subsection{Splitting \cn\ according to the values of $h$}

In this section collect some lemmas that will be useful in the proof of the other direction of \Tr{thPB}. The reader may choose to skip to \Sr{secMainProof} at this point and come back later.

We denote the set of bounded harmonic functions of \g by $BH(G)$.

\begin{lemma} \label{lab}
Let \g be a transient network. For every $h\in BH(G)$ and every $a<b\in \R$, the function $h_{[a,b)}$ defined by $z\mapsto \mu_{z}(\lim_n h(Z^n) \in [a,b))$ is harmonic and sharp.
\end{lemma}
\begin{proof}
To begin with, it is easy to check that the event $\{\lim_n h(Z^n) \in [a,b)\}$ is measurable. The assertion now follows from \Lr{sharph}.
\end{proof}

\begin{definition} \label{defab}
Let $(\cn,(\nu_z)_{z\in V})$ be an $M$-boundary that is faithful to every sharp harmonic function, and let $h\in BH(G)$. Recall that \fe\ measurable $X\subseteq \cn$, \ti\ a sharp function $s_X$ \st\ $f^{-1}(X)= 1^{s_X} \sydi \Phi$ where $\Phi$ is a null-set in path space \eqref{convfai}. Given a bounded interval $[a,b)\subset \R$, we define the \shf\ $y=y_{[a,b)}:= s_X \cap h_{[a,b)}$ where $h_{[a,b)}$ is as in \Lr{lab} and intersection as in \Lr{comb}. Since $\cn$ is faithful to $y$, \ti\ a set $Y\subseteq \cn$ \st\ $\mu(1^y \sydi f^{-1}(Y))=0$, and we let $X\restr_{[a,b)}$ denote such a set $Y$. By \Cr{hprob} we have
\labtequ{ref1}{$\nu_z(Y)= \mu_z(f^{-1}(Y)) = \mu_z(1^y) = y(z)$}

\end{definition}

\begin{lemma} \label{refine}
Let $(\cn,(\nu_z)_{z\in V})$ be an $M$-boundary that is faithful to every sharp harmonic function, and let $h\in BH(G)$. Let $a_0 < a_1 <\dots a_k < a_{k+1} \in \R$ be points such that the range of $h$ is contained in $(a_0,a_{k+1})$, then $\nu_z(X)= \sum_{0\leq i \leq k} \nu_z(X\restr_{[a_i,a_{i+1})})$ (hence $X$ equals $\bigcup X\restr_{[a_i,a_{i+1})}$ up to a null-set).
\end{lemma}
\begin{proof}
%


By \eqref{ref1} we have $\sum \nu_z(X\restr_{[a_i,a_{i+1})})= \sum_i y_{[a_i,a_{i+1})}(z)$.
By our definitions, we have 
\begin{align*}
y_{[a_i,a_{i+1})}(z)= s(z) \cap h_{[a_i,a_{i+1})} =\mu_z(1^s \cap 1^{h_{[a_i,a_{i+1})}})=\\ 
\mu_z(1^s \cap \{\lim_n h_{[a_i,a_{i+1})}(W^n)= 1\})
\end{align*}
where $s(\cdot):=\nu_\cdot(X)$. 
Now recall that $h_{[a_i,a_{i+1})}$ is a probability, namely of the event that \rw\ $W'^m$ from (the random vertex) $W^n$ will have its $h$ values converge to $[a_i,a_{i+1})$. Now note that the distribution of $W'^m$ is the same as that of the continuation of $W_n$ after the $n$th step. This means that $h_{[a_i,a_{i+1})}(W^n)$ equals the probability that $W^n$ itself displays this behaviour. Applying \Lzol\ \eqref{lzol} to the latter probability, we can thus deduce that the last expression above equals
$$\mu_z(1^s \cap \{\lim_n h(W^n) \in [a_i,a_{i+1})\}).$$
Plugging this into the above sum, we obtain 
$$\sum \nu_z(X\restr_{[a_i,a_{i+1})})= \sum_i \mu_z(1^s \cap \{\lim_n h(W^n) \in [a_i,a_{i+1})\}).$$

The latter sum however equals $\mu_{z}(1^s)$ since, by the  bounded martingale convergence theorem,  with probability 1 exactly one of the events $\{\lim_n h(W^n) \in [a_i,a_{i+1})\}$ occurs. Finally, recall that $\mu_{z}(1^s)= \mu_{z}(f^{-1}(X))= \nu_z(X)$ by \eqref{convfai} and the fact that $f$ is measure preserving.
\end{proof}

\subsection{Main proof of \Tr{thPB}} \label{secMainProof}

We proceed with the proof of the backward direction of \Tr{thPB}, that if $\cn$ is faithful to every \shf\ then it is a realisation of the Poisson boundary.

Let us start with the easier assertion we have to prove, namely that \fe\ $\hat{h}\in L^\infty(\cn)$ the function $f:V\to \R$ defined by $z\mapsto \int_\cn \hat{h}(\eta)d\nu_z(\eta)$ is bounded and harmonic.

It is immediate from its definition that the range of $f$ is contained in the range of $\hat{h}$, and so $f$ is bounded. The fact that $f$ is harmonic follows easily from \eqref{harmnu}.

\medskip
Next, we prove that \fe\ $h\in BH(G)$ \ti\ $\hat{h}: \cn \to \unin$ \st\ \fe\ $z\in V$ we have $h(z) = \int_\cn \hat{h}(\eta)d\nu_z(\eta)$, which is the core of \Tr{thPB}. Assume \obda\ that the range of $h$ is the interval $[0,1]$.

\medskip
Recall that \fe\ $0<a<b<1$ and any measurable $X\subseteq \cn$, we can define the measurable set $X \restr_{[a,b)}$ (\Dr{defab}). 
Using this we can,
\fe\ $z\in V$, and any measurable $X\subseteq \cn$, induce a measure $\nu_z^X$ on \unin\ by letting $\nu_z^X(I)=\nu_z(X \restr_I)$ \fe\ subinterval $I$ of \unin\ and extending to all Borel subsets of $[0,1]$ using Caratheodory's extension theorem. Now let
\labequ{H}{H_z(X):= \int_{\unin} a \nu_z^X(da).}
\note{$H_z(X)$ is the exp'on of $h$ subject to $X$.}
\medskip
It is straightforward to check that $H_z$ is a measure on \cn. Easily, $H_z$ is  uniformly continuous to $\nu_z$. Thus we can let 
$$R_z(\eta) = \frac{\partial H_z}{\partial \nu_z} (\eta)$$
be the corresponding Radon-Nikodym derivative. The range of $R_z$ is contained (up to a null-set) in the closure of the range of $h$. Thus $R_z \in L^\infty(\cn)$.

Recall that we would like to find $\hat{h}: \cn \to \unin$ \st\ \fe\ $z\in V$ we have $h(z) = \int_\cn \hat{h}(\eta)d\nu_z(\eta)$. Thus it suffices to prove the following two claims. 

\medskip
{\bf Claim 1:} \Fe\ $z,o$ we have $R_z(\eta)=R_o(\eta)$ for almost every $\eta$.

\medskip
This allows us to define $\hat{h}:= R_o$ for a fixed $o\in V$. Recall that $R_o \in L^\infty(\cn)$.

\medskip
{\bf Claim 2:} \Fe\ $z\in V$ we have $h(z) = \int_\cn dH_z = \int_\cn R_z(\eta) d\nu_z(\eta)= \int_\cn \hat{h}(\eta) d\nu_z(\eta)$.
\medskip

To prove Claim 1, suppose to the contrary that \ti\ $X\subseteq \cn$ of positive measure \st\ $R_z(\eta)>R_o(\eta)+\eps$ for some $\eps>0$ and every $\eta\in X$. By \Lr{refine}, we can decompose $\cn$ into a union $\bigcup_{0\leq i\leq k} \cn\restr_{[a_i,a_{i+1})}$ of measurable subsets, with $a_0=0$ and $a_{k+1}=1$, where we are free to choose the $a_i$ as we wish. So let us choose them in such a way that $a_{i+1}-a_i<\eps$ \fe\ $i$.

Now note that, by the definition of $H_z$, we have $\frac{H_z(\cn\restr_{[a_i,a_{i+1})})}{\nu_z(\cn\restr_{[a_i,a_{i+1})})} \in [a_i,a_{i+1})$; indeed, if $I\cap J=\emptyset$ then $\nu((\cn\restr_I)\restr_J)=0$, and so $\nu_z^{\cn\restr_{[a_i,a_{i+1})}}$ is supported on $[a_i,a_{i+1})$. Even more, if $Y$ is any measurable subset of $\cn\restr_{[a_i,a_{i+1})}$ of positive measure, we also have $\frac{H_z(Y)}{\nu_z(Y)} \in [a_i,a_{i+1})$. Thus $R_z(\eta)\in [a_i,a_{i+1})$ for almost every $\eta\in \cn\restr_{[a_i,a_{i+1})}$. But as this holds \fe\ $z$, and $a_{i+1}-a_i<\eps$, this means that $|R_z(\eta)-R_o(\eta)|<\eps$ for almost every $\eta\in \cn$, contradicting the existence of $X$ as defined above. This proves Claim 1.

\medskip
To prove Claim 2, we have to show that $h(z)= \int_\cn dH_z = H_z(\cn) :=\int_{\unin} a \nu_z^\cn(da)$. Since $h$ is harmonic, we have  $h(z)= \int_{b\in V} h(b) \mu_z^n(b)$ \fe\ $n$, where $\mu_z^n$ denotes the distribution of the $n$th step of \rw\ from $z$. Easily, the latter sum equals $ \int_{\unin} a \mu_z^n(da)$ by a double-counting argument, where, with a slight abuse of notation, we 
treat $\mu_z^n$ as a probability measure on \unin\ by making the convention $\mu_z^n(da):= \mu_z^n(\{b\in V \mid h(b)\in da\})$. 
Comparing the latter integral with the one above, we see that it suffices to find a family $\ci$ of intervals of \unin\ that is a basis for its topology and $\lim \mu_z^n(da) = \nu_z^\cn(da)$ \fe\ interval $da \in \ci$.

For this, call a number $a\in \R$  $h$-singular, if $\mu_o(\lim_n h(Z_n)=a)>0$. Let $\ci$ be the family of intervals contained in \unin\ the endpoints of which are not $h$-singular. Since \ta\ at most countably many 
$h$-singular points, $\ci$ is clearly a basis of \unin.

To see that \fe\ $da \in \ci$ we have $\lim \mu_z^n(da) = \nu_z^\cn(da)$, recall that $\nu_z^\cn(da)= \nu_z(\cn \restr da)= \mu_z(\{\lim f(W^n)\in da\})$ where we used \eqref{ref1}, and note that $\mu_z(\{\lim f(W^n)\in da\})\leq \liminf_n \mu_z^n(da)$ because, as the endpoints of $da$ are not $h$-singular, subject to $\{\lim f(W^n)\in da\}$ our random walk almost surely visits $da=\{b\in B^n \mid h(b)\in da\}$ for almost every $n$. We claim that, conversely,  $\nu_z^\cn(da)\geq \limsup_n \mu_z^n(da)$. To see this, note that $\mu_z^n(da)+\mu_z^n(da^c)=1$, where $da^c$ is the complement $\unin\sm da$ of $da$, because  $\mu_z^n$ is a probability measure. By \Lr{refine} we have $\nu_z^\cn(da)+\nu_z^\cn(da^c)=1$ as well. Thus, applying the above arguments to $da^c$ instead of $da$, which yields $\nu_z^\cn(da^c) \leq \liminf_n \mu_z^n(da^c)$, we conclude that $\nu_z^\cn(da)\geq \limsup_n \mu_z^n(da)$ as claimed. This means that $\lim \mu_z^n(da)$ exists and equals $\nu_z^\cn(da)$. This proves Claim 2, completing the proof of the existence of the desired function $\hat{h}$.


\medskip
It remains to check that $\hat{h}$ is unique up to modification on a null-set. If this is not the case, then  \ti\ another candidate $\hat{h'}$ \st, \fe\ $z\in V$, we have 
%
$$h(z) = \int_\cn \hat{h}(\eta)d\nu_z(\eta)= \int_\cn \hat{h'}(\eta)d\nu_z(\eta).$$
Define the function  $\hat{k}(\eta):= \hat{h}(\eta)-\hat{h'}(\eta)$ on $\cn$, and note that\\  
$k(z):= \int_\cn \hat{k}(\eta)d\nu_z(\eta)=0$ \fe\ $z$.
Now if $\hat{h}$ does not coincide with $\hat{h'}$ $\nu$-almost everywhere, \ti\ some $\eps>0$ \st\ the measurable set $X:= \{ \eta\in \cn \mid \hat{k}(\eta)> \eps \}$ is not a null-set. 

Let $s_X(z):= \nu_z(X)=\mu_z(f^{-1}(X))$, which is a \shf\ by \Lr{sharph}. Thus, by \Cr{close1}, \ti\ $x\in V$ with $s_X(x)=\nu_x(X)> 1-\eps'$ for any $\eps'>0$ we choose. But this means that $k(x) > \eps (1-\eps') - \eps \inf \hat{k}$ by the choice of $X$. Choosing $\eps'$ large enough compared to $\eps \inf \hat{k}$, we obtain a contradiction to $k=0$ that completes the proof.

\section{Construction of the tiling} \label{constr}

In this section we show how any plane transient graph can be associated with a tiling of the cylinder $K := \R/\Z \times [0,1]$ with squares, or rectangles if the edges of \g have various resistances. Our construction follows the lines of \cite{BeSchrST}, but we will be pointing out many properties of this tiling that we will need later. Thus this section could be useful to the reader already acquainted with \cite{BeSchrST}.

\subsection{The Random Walk flow} \label{secRWF}

Fix a vertex $o\in V$ and \fe\ vertex $v\in V$ let $h(v)$ be the probability $p_v(o)$ that random walk from $v$ will ever reach $o$. Thus $h(o)=1$. We will use $h(v)$ as the `height' coordinate of $v$ in the construction of the \squt\ in the next section.

Recall that the \defi{Green function} $G(x,y)$ is defined as the expected number of visits to $y$ by \rw\ from $x$. Let 
$$h'(v):= \frac{\pi_o G(o,v)}{G(o,o) \pi_v} ,$$
where as usual $\pi_x:= \sum_{y\in  N(x)} c(xy)$.

We claim that 
\labtequ{hh}{$h'(v)=h(v)$.}
Indeed this is a consequence of the reversibility of our \rw: it is well-known, and not hard to prove (see \cite[Exercise~2.1]{LyonsBook}), that $\pi_v G(v,o) = \pi_o G(o,v)$. Observing that $G(v,o)= p_v(o) G(o,o)$ now immediately yields \eqref{hh}.

It is no loss of generality to assume that the constant $\frac{\pi_o}{G(o,o)}$ appearing in the definition of $h'$ equals 1: multiplying the conductances $c$ by a constant does not affect the behaviour of our random walk, and hence $G(o,o)$, and so we can achieve $\pi_o=G(o,o)$ by multiplying with the appropriate constant. Thus, from now on we can assume that
\labtequ{defh}{$h(v) = \frac{G(o,v)}{\pi_v} =\frac1{\pi_v} \ExI{\text{\# of visits to $v$ by \rw\ from $o$}}$.}
A \defi{directed edge} $\ded{xy}$ of \g is an ordered pair $(x,y)$ of vertices \st\ $\{x,y\}$ is an edge of $G$. Define the \defi{\rw\ flow} to be the function $w(\are)$ on the set of directed edges $\are$ of \g equaling the expected net number of traversals of $\are$ by \rw\ from $o$:
$$w(\ded{xy}):= G(o,x) p_{x\to y} - G(o,y) p_{y\to x}.$$
Since our \rw\ is transient, $w(\are)$ is always finite.
Note that, by \eqref{defh} and \eqref{trpr}, we have 
\labtequ{wch}{$w(\ded{xy})= c(xy)(h(x) - h(y))$.}
In electrical network terminology, \eqref{wch} says that the pair $h,w$ satisfies Ohm's law. Thus $w$ is `antisymmetric', i.e.\ $w(\ded{xy})= -w(\ded{yx})$. Moreover, $w$ is a \defi{flow} from $o$ to infinity, by which we mean that it satisfies the following conservation condition, known as \knl, at every vertex $x$ other than $o$:
\labtequ{k1}{$w^*(x):= \sum_{y\in  N(x)} w(\ded{xy}) = 0$,}
where $N(x)$ is the set of neighbours of $x$. This is equivalent to saying that $h$ is harmonic at every vertex except $o$.

The following fact can be found in \cite[Exercise~2.87]{LyonsBook}. 
\begin{lemma} \label{hto0}
For almost every \rw\ $(Z^n)$, we have $\lim_n h(Z^n)=0$.
\end{lemma}
	\comment{\begin{proof}
By Doob's first martingale convergence theorem, $\lim_n h(Z^n)$ exists almost surely as $h$ is a bounded supermartingale. 

	\end{proof}
}

Note that up to now we did not use the planarity of \G, so all above statements hold for an arbitrarily transient graph.

\subsection{The dual graph $G^*$} \label{secDual}

Let $G=(V,E)$ be a planar graph, and fix a proper embedding of \g into the plane $\R^2$. Our graph can now be considered as a \defi{plane graph}, in other words, $V$ is now a subset of $\R^2$ and $E$ a set of arcs in $\R^2$ each joining two points in $V$. It is a standard fact that one can associate with \g a further plane graph $G^*= (V^*,E^*)$, called the (geometric) \defi{dual} of \G, having the following properties: 
\begin{enumerate}
 \item Every face of \g contains precisely one vertex of $G^*$ and vice versa;
\item There is a bijection $e\mapsto e^*$ from $E$ to $E^*$ \st\ $e\cap G^*= e^*\cap G$ consists of precisely one point, namely a point at which the edges $e,e^*$ meet.
\end{enumerate}
Note that if $G^*$ is a geometric dual of \g then \g is a geometric dual of $G^*$, but we will not need this fact.  For example, the geometric dual of a hexagonal lattice is a triangular lattice with 6 triangles meeting at every vertex. See \cite{diestelBook05} for more details on dual graphs.

The orientability of the plane allows us to extend the bijection $e\mapsto e^*$ to a bijection between the directed edges of $G$ and $G^*$ in such a way that if $C$ is a directed cycle of \g and $\arE(C)$ its set of directed edges, then $\{\are^* \mid \are\in \arE(C) \}$ is a \defi{directed cut}, that is, it coincides with the set of edges from a subset $A$ of $V$ to $V\sm A$, directed from the endvertex in $A$ to the endvertex in $V\sm A$ (where $A$ is the set of vertices of $G^*$ contained in one side of $C$).


\subsection{The Tiling} \label{stil}	

Given a plane, transient, and {uniquely absorbing} graph $G$, we now construct a tiling of the cylinder $K := \R/\Z \times [0,1]$ by associating to each edge of $N$ a rectangle $R_e\subseteq K$, with sides parallel to the boundary of $K$.  

In order to specify $R_e$ we will use four real coordinates, corresponding to the endvertices $x,y$ of $e$ in $G$ and the endvertices $x',y'$ of $e^*$ in the dual $G^*$: two of these coordinates will be the values $h(x),h(y)$, where  
$h$ is the function from \Sr{secRWF} defined by means of the random walk flow; these values will be used as `height' coordinates. We now specify a `width' function $w$ on the vertices of $G^*$ to be used for the other coordinates.

Fix an arbitrary vertex $\zeta$ of $G^*$ and set $w(\zeta)=0$. For every other vertex $z\in V(G^*)$, pick a \pth{z}{\zeta}\ $P_z= z_0 z_1 \ldots z_k$, where $z_0=z$ and $z_k=\zeta$, and let $w(z)= \sum_{i<k} w(\ded{z_i z_{i+1}}^*) \mod 1$, where $\ded{z_i z_{i+1}}^*$ denotes the directed edge $\are$ of \g such that $\are^*= \ded{z_i z_{i+1}}$ (recall the remark on orientability at the end of the previous section).

This value $w(z)$ does not depend on the choice of $P_z$, but only on the endpoint $z$; this fact is a consequence of a well-known duality between \knl\ on a plane network \g and \kcl\ on its dual $G^*$. To be more precise, if $C= \ded{e_1} \ded{e_2} \ldots \ded{e_k}$ is a directed cycle in $G^*$ \st\ the set of vertices $U$ of \g contained in one of the sides of $C$ is finite, then $\sum {w(\ded{e_i})} = \sum_{x\in U} w^*(x)$. Using this, and the fact that \g is uniquely absorbing, it is not hard to check that the latter sum equals a multiple $k \eta$ of the total flow $\eta:= w^*(o)$ out of $o$, where $k$ is the `winding number' of $C$ around $o$; see \cite[Lemma~3.2]{BeSchrST} for a detailed proof\note{uniquely absorbing has to be used to make sure that the cut defined by a height parallel cycle  is a closed walk in $G^*$ rather that a union of such. It is not needed for finiteness of the cut}. As we are assuming that $\eta=1$, our definition of $w(z)$ does indeed not depend on  the choice of $P_z$.
\medskip

Having defined $w: V(G^*) \to \R / \Z$, we can now specify the rectangle $R_e$ corresponding to an edge $e$ as above in our tiling: $R_e$ is one of the two rectangles in $K$ bounded between the horizontal lines $h=h(x)$ and $h=h(y)$ and the perpendicular lines $w(x')$ and $w(y')$. To decide which of the two, orient $e$ from its endvertex of lower $h$ value into the one with higher value, recall that this induces an orientation of $e^*$, and choose that rectangle in which the $w$ values increase as we move from the initial vertex of $e^*$ to the terminal vertex inside the rectangle\note{check this}.

\medskip
This completes the definition of the tiling, which from now on we denote by $T= T_{G,o}$. Formally, $T$ can be defined as a function on $E$ mapping each $e$ to a rectangle $R_e\subset K$, but it will be more convenient in the sequel to assume $T$ to be a function on \G, seen as a metric space (recall the discussion in \Sr{prelim}), mapping any interior point $p$ of an edge $e$ to the maximal horizontal interval at height $h(p)$ contained $R_e$, and any vertex $x$ to the maximal horizontal interval at height $h(p)$ contained in $\bigcup_{x\in e} R_e$. 


Note that $T(G)$ does not meet the base $\cc := \R/\Z \times \{0\}$ of our cylinder $K$. It is the main aim of this paper to show that $\cc$, to be thought of as the \defi{boundary} of $K$, is a realisation of the Poisson boundary of \G.

Let us point out some further properties of our tiling $T$.
By \eqref{wch}, the aspect ratio of $R_e$ equals the conductance of $e$:
\labtequ{ar}{$\frac{w(e)}{dh(e)} = c(e)$.}
In particular, if $c$ is identically $1$ then we obtain a square tiling.

Let $\arE(x)$ be the set of directed edges emanating from vertex $x$. Let $E^+(x)\subseteq \arE(x)$ be the set of directed edges $\ded{xy}$ with $h(y)\geq h(x)$ and $E^-(x)$ be the set of directed edges $\ded{xy}$ with $h(y) < h(x)$. Note that the \rw\ flow flows into $x$ along the edges in $E^+(x)$ and out of $x$ along the edges in $E^-(x)$.

For a vertex $x$, define $w(x)$ to be the \defi{width} of its image in $T$; that is, 
$w(x):= \sum_{\are\in E^-(x)} w(\are) = \sum_{\are\in E^+(x)} w(\are) = 1/2 \sum_{\are \in E(x)} |w(\are)|$ ($w(x)$ should not be confused with $w^*(x)$). By the definitions, we have
\labtequ{wv}{$w(x)= \ExI{\text{net \# of particles arriving to $x$ from above} } $.}

\note{
\begin{lemma}\label{finfcut}
For every $h>0$ there are only finitely many vertices $x$ with $h(x)\geq h$ and $w(x)>0$.
\end{lemma}
\begin{proof}

\end{proof}
}


The energy of a function $u: V \to \R$ is defined by\\ 
$E(u):= \sum_{xy\in E(G)} (u(x)-u(y))^2$ (or more generally, $\sum_{xy\in E(G)} (u(x)-u(y))^2/c(xy)$ if the conductances are non-constant). Note that for our height function $h$, $E(h)$ equals the area of $K$, which is 1, as the contribution of each edge $e$ to $E(h)$ is the area of $R_e$ by definition. This, combined with the bounded degree condition, implies

\begin{lemma}\label{bodeg}
If $\pi$ is bounded from above then $w$ converges to 0, i.e.\
\labtequ{limax}{$\lim_n w(x_n)=0$ \fe\ enumeration $(x_n)$ of $V$.}
\end{lemma}
\begin{proof}
If \ta\ infinitely many $x_i$ with $w(x_i)\geq \eps>0$, then each of them is incident with an edge $e_i$ with $w(e_i)> \eps/D$ where $D$ is the maximal degree. But each such edge contributes more than $(\eps/D)^2$ to $E(h)$, contradicting the fact that the latter is finite.
\end{proof}

We formulated \Tr{main} for graphs of bounded degree only, but most of our proof applies to all graphs satisfying the weaker condition \eqref{limax}, the only exception being the part on convergence to the boundary (\Sr{convBou}).

\section{Probabilistic interpretations of the geography of the tiling  cylinder} \label{geo}

Our next observation is that if we modify our graph by subdividing some edge $e$ into two edges the total resistance of which equals the resistance $1/c(e)$ of $e$, then the tiling remains practically unchanged. Note that in our metric space model of \G\ introduced in \Sr{prelim}, such an operation is tantamount to declaring some interior point $x$ of the arc $e$ to be a vertex while leaving the metric unnchanged.

\begin{lemma}\label{dummy}
Let \g be a transient plane graph and $x$ be an interior point of an edge $e$ of \G, and let $G'$ be the plane graph obtained from $G$ by declaring $x$ to be a vertex. Then $T_o(G')$ can be obtained from $T_G$ by cutting the rectangle $R_e$ into two rectangles along the horizontal line $h(x)$.
\end{lemma}
\begin{proof}
Observe that the functions $h$ and $w$ of \Sr{secRWF} remain unchanged at every old vertex or edge. This can be checked directly, or by using the fact that our random walk can be obtained as the sequence of distinct vertices visited by brownian motion on $G$, defining $h(x)$ to be the probability that brownian motion from any point $x$ of $N$ will ever reach $o$, noticing that this definition of $h$ coincides with that of \Sr{secRWF} on $V$, and observing that brownian motion cannot tell the difference between $G$ and $G'$.

The assertion now is an immediate consequence of the construction of the tilings.
\end{proof}

Another way of stating \Lr{dummy} is that if we consider our tiling $T$ to be a function on \g as described above, then declaring interior edge points to be vertices does not change $T$.

\note{
	\begin{lemma}\label{noback}	
For a vertex $r$ and $a\in (h(r),1)$ let $P_r(a)$ denote the probability that \rw\ from $r$ will ever visit the set of vertices $x$ with $h(x)\geq a$. Then 
$$P_r(a)\leq h(r)/a.$$
\end{lemma}
\begin{proof}
Recall that $h(v)$ is the probability that random walk from $v$ will ever reach $o$. Applying this to both $r$ and the $x$ in the definition of $P_r(a)$, and using the Markov property of \rw, we obtain $h(r)\geq P_r(a) a$ as claimed.
\end{proof}
}

\subsection{Parallel circles} \label{secPar}

We are now going to consider a sequence \seq{G} of subgraphs of \G\ with nice properties using our tiling $T$ of the cylinder $K$: let \seq{l}, $0<l_n<1$, be a sequence of numbers monotonely converging to 0, define the \defi{parallel circle} $L_n$ to be the set of points of $K$ at height $l_n$, and let $G_n$ be the subgraph mapped by $T$ to the strip between the top of the cylinder $K$ and $L_n$; in other words, $G_n$ is the subgraph spanned by the vertices $x$ with $h(x)\geq l_n$. Similarly, for $m<n\leq \infty$ define the \defi{strip} $G_n^m$ to be the subgraph of \g (and $G_n$) spanned by the vertices $x$ with $l_m\geq h(x) \geq l_n$, where we set $l_\infty= 0$. Note that $V(G_n) \cup V(G^n_\infty)= V(G)$.

We may, and will, assume that each $L_n$ meets \g at vertices only (thus $G_n \cup G^n_\infty= G$ and $G_n \cap G^n_\infty= L_n$); for if $e$ is an edge dissected by $L_n$, then we can, by \Lr{dummy}, put a dummy vertex at the point $x$  in the  interior of $e$ with $h(x)=l_n$. Let $B_n$ be the preimage of $L_n$, i.e.\ the set of (possibly dummy) vertices $b$ with $h(b) = l_n$. We claim that 
\labtequ{Bn}{\rw\ from $o$ visits every $B_n$ almost surely.}
Indeed, by \Lr{hto0} the heights of the vertices visited by \rw\ almost surely converge to 0; in other words, \rw\ converges to the base of $K$. Thus it almost surely visits every parallel circle $L_n$. Note that \eqref{Bn} is trivially true if each $G_n$ is a finite graph, which is the case for a large class of planar graphs \G\ but not always. 

The advantage of our assumption that each $L_n$ meets \g at vertices only can be seen in the following lemma.
Denote by $\mu^n_o(b)$ the exit distribution (harmonic measure) of \rw\ from $o$ killed at $B_n$.

\begin{lemma}\label{wn}
\Fe\ $n$ and every $b\in B_n$, $w(b)=\mu^n_o(b)$.
\end{lemma}
%
%




This important observation does not presuppose \g to be planar; in the planar case we can visualise $w$ as width \wrt\ $T$, but $w$ is well-defined, and \Lr{wn} true, for any transient graph. Note that \Lr{wn} is a reformulation of \Or{obsbm} when considering the brownian motion mentioned in \Sr{prelim}; the advantage of that formulation is that it is not affected by the dummy vertices we introduced above (recall the discussion in the proof of \Lr{dummy}).

We will also show that $w(b),b\in B_n$ also coincides with the distribution of the last vertex of $B_n$ visited by \rw\ from $o$ (and so the distribution of the first vertex of $B_n$ visited coincides with that of the last). 

Before we prove all this, we need to introduce a further tool.


If \g is any finite (or recurrent) planar graph, then it is possible to construct a tiling as we did in the infinite case, except that we now have to stop the \rw\ at some point because if we do not, then $h(v)$ (the probability to reach $o$ from $v$) will be 1 \fe\ $v$, yielding a trivial tiling. One possibility is to stop our \rw\ upon its first visit to a fixed vertex $o'$; this is in fact what Brooks et.\ al.\ did when they first introduced square tilings \cite{BeSchrST,BSST}. 

We will follow a slightly different approach: we will construct a tiling $T_{G_n}$ corresponding to $G_n$ by starting our \rw\ at $o$ and stopping it upon its first visit to the set $B_n$ (recall that such a visit exists almost surely \eqref{Bn}). For this, we repeat the construction of \Sr{constr} except that we now define $h_n(v)$ to be the probability for \rw\ from $v$ to reach $o$ before $B_n$. Note that now we can reformulate \eqref{defh} as \\ 
$h_n(v) = \frac1{\pi_v} \ExI{\text{\# of visits to $v$ by \rw\ from $o$ stopped at $B_n$}}$.\\
The rest of the construction of the tiling remains the same. Note that we now have $h_n(v)=0$ \fe\ $v\in B_n$; in fact this tiling $T_{G_n}$ can be obtained from our tiling $T_G$ of \g by linearly stretching the part of $T_G$ bounded between the lines $h=l_0$ and $h=l_n$ to bring the latter line to height zero:

\begin{lemma} \label{stretch}
$T_{G_n}$ can be obtained from $T_G$ by the following transformation: if $T_G$ maps a point $p$ of $G_n$ to the interval $I_p \times \{h(p)\}$, then $T_{G_n}$ maps $p$ to the interval $I_p \times \{ \frac{h(p)-l_n}{1-l_n}\}$. 
\end{lemma} 
\begin{proof}
Let $f(v):= \frac{h(v)-l_n}{1-l_n}$. Note that the formula $h_n(v) = f(v)$ is trivially correct for $v\in B_n \cup \{o\}$, and that both $h$ and $h_n$ are harmonic on the rest of $G_n$. Moreover,  $f$ is harmonic since $h$ is, because the former is a linear transformation of the latter. Recall that \rw\ from $o$ visits $B_n$ almost surely \eqref{Bn}. Thus, since the values of a harmonic function are uniquely determined by its boundary values (recall \eqref{harmB} and the remark after it), $h_n$ must coincide with $f$ everywhere.

Recall that $w$ is uniquely determined by $h$ \eqref{wch}. Thus the width of each vertex is the same in the two tilings, and as we are using the same embedding in both cases, the position of the interval corresponding to any vertex is also the same.
\end{proof}

\Lr{stretch} easily implies \Lr{wn}: recall that $w(b)$ equals the net number of particles ariving to $v$ from above by \eqref{wv}. But in the killed \rw\ we used in the construction of $T_{G_n}$, this number coincides with the exit probability $\mu^n_o(b)$. Since  by \Lr{stretch} $w(b)$ is the same in $T_{G_n}$ and $T_G$, \Lr{wn} follows.

\medskip
For $m\in\N$ and a walk $W$ from $o$, we define an \defi{$m$-subwalk} of $W$ to be a maximal subwalk of $W$ that has no interior vertex in $B_m$. Note that $W$ is the concatenation of all its $m$-subwalks. The first of them is called the \defi{initial} $m$-subwalk, and if \ti\ a last one it is called the \defi{final} $m$-subwalk; all others are \defi{interior} $m$-subwalks, and start and end in $B_m$ but do not visit $B_m$ in between.

Recall that $w(\are)$ was defined as the expected net number of traversals of $\are$ by \rw\ from $o$. Let $w^m(\are)$ be the contribution to $w$ by the initial and final $m$-subwalks of \rw\ from $o$, i.e.\ we let $W$ be a \rw\ from $o$ and let $w^m(\are)$ be the expected net number of times that the initial $m$-subwalk of $W$ goes through $\are$ plus the expected net number of times that the final $m$-subwalk of $W$ goes through $\are$. We define $w^m_n(\are)$ similarly for $n>m$, except that $W$ is stopped upon its first visit to $B_n$.

\begin{lemma} \label{ww}
$w(\are)=w^m(\are)=w^m_n(\are)$ \fe\ $\are\in \arE(G)$.
\end{lemma}
\begin{proof}
By the definitions, an equivalent statement is that \fe\ directed edge $\are$, the interior $m$-subwalks of \rw\ from $o$ traverse $e$ in each direction equally often in expectation, where \rw\ is stopped upon its first visit to $B_n$ when considering the second equation.

To prove that these traversals indeed cancel out, let $Q$ be a walk that starts and ends in $B_m$ and does not visit $B_m$ in between, in other words, a candidate 
for an interior $m$-subwalk of our \rw; for the proof of $w=W_n^m$ we also demand $Q$ to avoid $B_n$ here. Let $a,z\in B_m$ be the endvertices of $Q$. Let $p_{az}$ denote the probability that \rw\ from $a$ will coincide with $Q$ up to its first re-visit to $B_m$, and conversely, let $p_{za}$ denote the probability that \rw\ from $z$ will coincide with the inverse $Q^-$ of $Q$ up to its first re-visit to $B_m$. It is well-known, and straightforward to prove using \eqref{trpr}, that 
$\pi(a)p_{az}= \pi(z)p_{za}$. 

Now recall that $h(a)=  G(o,a)/\pi(a)$ by \eqref{defh}, and similarly $h(z)=  G(o,z)/\pi(z)$, where $G(o,a)$ denotes the expected number of visits to $a$. Since $a,z\in B_m$, we have $h(a)=h(z)$. Now note that the expected number of times that an $m$-subwalk of our \rw\ coincides with $Q$ equals $G(o,a)p_{az}$ by linearity of expectation and the Markov property. Similarly, the expected number of times that an $m$-subwalk of our \rw\ coincides with $Q^-$ is $G(o,z)p_{za}$. Putting all this together we have
$$G(o,a)p_{az} = {h(a)}{\pi(a)} p_{az} = h(a) \pi(z)p_{za} = h(z) \pi(z)p_{za} =  G(o,z)p_{za}.$$
Thus, the contribution of the pair of walks $Q,Q^-$ to $w(\are)$ is zero, for even if they contain the edge $\are$, their contributions cancel out in expectation. Since all walks that are candidates
for an interior $m$-subwalk of our \rw\ can be organised in such pairs, their overall contribution to $w(\are)$ is zero as claimed.
\end{proof}

This implies the following assertion, which complements \Lr{wn}.
\begin{corollary}\label{wnlast}
\Fe\ $m$ and every $b\in B_m$, $w(b)$ coincides with the probability distribution of the last visit of \rw\ (from $o$) to $B_m$. This remains true if the \rw\ is stopped upon its first visit to $B_n$ for $n> m$.
\end{corollary}
(The reader might be upset at this point for our use of the letter $m$ here and the letter $n$ in \Lr{wn}, but will be grateful for this when reading \Sr{hard}.)
\begin{proof}
The difference between the distributions of the first and last visits to $B_m$ is determined by the behaviour of the interior $m$-subwalks of \rw\ from $o$. But by \Lr{ww} the influences of these $m$-subwalks cancel out, and so \Lr{wn} implies our claim. For the second sentence we use \Lr{stretch}.
\end{proof}

\note{not needed:
	\begin{lemma}\label{wn2}
For \rw\ on $G$ started at a random vertex in $B_m$ according to the  distribution $w(b),b\in B_m$, the  exit distribution to $B_n$ for any $n> m$ coincides with $w(b),b\in B_n$. The same holds if the \rw\ takes place on $G^m_\infty$ rather than \G.
\end{lemma}
\begin{proof}
For the first assertion (\rw\ on \G), start a \rw\ at $o$, and stop it upon visiting $B_n$. Suppose that this \rw\ is paused at its first visit to $B_m$, and at this moment an observer enters the room and then the \rw\ is resumed. What this observer experiences, is a \rw\ on \g started at $B_m$ and killed at $B_n$. Moreover, the starting distribution he experiences is $w(b),b\in B_m$ by \Lr{wn} applied on $B_m$, and the exit distribution he  
experiences coincides with that of the whole \rw, which is $w(b),b\in B_n$ again by \Lr{wn}, this time applied on $B_n$.

For the second assertion (\rw\ on $G^m_\infty$), we can use the argument of the proof of \Lr{wn} again, namely that all the traffic between vertices 
\end{proof}
} 

\note{not needed:
\begin{lemma}\label{rev}
$\pi(a)p_a(b) = \pi(b)p_b(a)$
\end{lemma}
\begin{proof}
reversibility.
\end{proof}
}


\note{not needed; same as \Lr{wn2}:
\begin{corollary}
For $m<n\in \N$, start a \rw\ at a random vertex of $B_m$ according to the distribution $w:B_m \to [0,1]$, and stop it at its first visit to $B_n$. Then the probability distribution of the final vertex coincides with $w: B_n \to [0,1]$.
\end{corollary}
\begin{proof}
\end{proof}
} 

\subsection{Meridians}
Having studied the nice probabilistic behaviour of parallel circles, we now turn our attention to \defi{meridians}, i.e.\ lines in $K$ with a constant width coordinate. We will prove an assertion which is, in a sense, dual to \Lr{ww}: for every meridian $M$, the expected number of particles crossing $M$ from left to right equals the  expected number of particles crossing $M$ from right to left.

Some care needs to be taken before we can make such an assertion, since if our \rw\ is currently at a vertex $x$ the span $T(x)$ of which (recall that $T(x)$ is a horizontal interval in $K$) is dissected by $M$, then we cannot say whether the particle is on the right or left of $M$. To amend this, we assume that at each step $n$ of our \rw, an additional random experiment is made to choose a random point $P^n$ in the span $T(Z^n)$ of the current position $Z^n$ uniformly among all points in  $T(Z^n)$, and all these experiments are independent from each other and $Z^i, i<n$, and we think of  $P^n$ as the position of our \rw\ at step $n$. With this assumption in mind, we can now state the following

\begin{lemma}\label{mer}
\Fe\ vertex $x$ and every meridian $M$ meeting $T(x)$, the expected number of particles crossing $M$ from left to right at $x$ equals the expected number of particles crossing $M$ from right to left at $x$.
\end{lemma}
\begin{proof}
To begin with, suppose for simplicity that $M$ does not dissect any square $R_e\subset K$ associated to an edge $e$ incident with $x$. Thus locally the position of $M$ is like in the left of \fig{mers}. By the construction of $T$, this implies that for the two vertices $r,s$ of $G^*$ lying in the faces of $G$ separating the edges mapped to the left of $M$ from those mapped to its right we have $w(r)=w(s)=w(M)$, where $w(M)$ denotes the common width coordinate of the points in $M$ (\fig{mers} right). 

\showFig{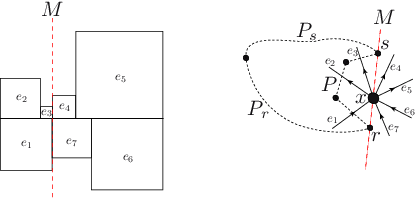}{The local situation at a vertex $x$ whose span $T(x)$ is dissected by a meridian $M$ in the tiling and the graph.}

Now let $P$ be the \pth{r}{s}\ in  $G^*$ comprising the set of edges $E^r$ incident with $x$ on the `left' of $M$. Recall that the values $w(r),w(s)$ where specified by choosing arbitrary paths $P_r,P_s$ from $r,s$ respectively to the reference vertex $\zeta$ of $G^*$, and adding the $w(\are)$ values along such a path. Now note that we can obtain a candidate for $P_r$ by prefixing $P_s$ by $P$. But since $w(r)=w(s)$, this implies that $\sum_{\are \in P} w(\are)=0$.  The definition of the \rw\ flow $w$ now implies that the net flow into $x$ along edges in $E^r$ is zero. Likewise, the net flow into $x$ along the `right' edges $E^l$ is zero. Our claim that the expected number of particles crossing $M$ from left to right at $x$ equals the expected number of particles crossing $M$ from right to left now follows from \knl\ \eqref{k1} and the Sand-bucket Lemma below (brown sand corresponds to particles coming to $x$ from the left; shovel them to B if they leave $x$ from the right).

It is straightforward to adapt this argument to the general case where $M$ does dissect some square $R_e$; we leave the details to the reader.
\end{proof}

\begin{lemma}[Sand-bucket Lemma]\label{bucket}
Alice has two (not necessarily equally full) buckets of sand A and B, where A contains only brown sand and B only white sand. If she puts one shovelful from A to B, mixes arbitrarily, then puts back one shovelful from B to A, then the amount (measured by volume) of white sand in A equal the amount of brown sand in B.
\end{lemma}

\comment{
	\begin{lemma}\label{mer}
For every meridian $M$, the expected number of visits to the vertex set of $M$ is finite.
\end{lemma}
\begin{proof}

\end{proof}

\begin{corollary}\label{merwh}

\end{corollary}
\begin{proof}
	\end{proof}
}

\subsection{Convergence of \rw\ to the boundary} \label{convBou}

In this section we provide a proof of the almost sure convergence of the image of the trajectory or \rw\ under $T$ to a point in the boundary $\cc$ of $K$ which is maybe simpler than the proof of \cite{BeSchrST}. This proof is the only occasion in this paper where the bounded degree condition cannot be replaced by the weaker \eqref{limax}. Let $D:= \max_{x \in V} d(x)$ be the highest degree in \G.

To begin with, recall that the image of our \rw\ converges to the boundary $\cc$ of $K$ by \Lr{hto0} (we do not need the bounded degree condition for this). So it just remains to show that the width coordinates converge too.

For this, let $X$ be an interval of $\cc$, and let $M_L,M_R$ be the two meridians corresponding to the endpoints of $X$. Let $L$ be the set of vertices $x$ such that $M_L$ intersects $T(x)$, and define $R$ similarly. Note that unless $M_L,M_R$ happen to meet some vertex at an endpoint of its span ---a possibility we can exclude since there are countably many such meridians--- $L,R$ are the vertex sets of two 1-way infinite paths of \G.

We claim that the expected number of times that our \rw\ alternates between $L$ and $R$ is finite. This implies that almost surely the number of such alternations is finite, and as this holds \fe\ $X$, convergence to the boundary follows. 

It is not hard to see that the expected number of times that \rw\ from $o$ goes from $L$ to $R$ equals  $\sum_{x\in L} G(o,x) \sum_{y\in R} p_{xy}$, where as usual $ G(o,x)$ denotes the expected number of visits to $x$, and $p_{xy}$ denotes the probability that \rw\ from $x$ exits $L \cup R$ at a given vertex $y\in R$. Thus, we can write the above claim as follows: 
\labtequ{Gp}{$\sum_{x\in L,y\in L}  G(o,x)  p_{xy}< \infty$.}
We are going to prove this exploiting the relations between \rw s and electrical networks.

Recall that $ G(o,x)= h(x) \pi_x$ \eqref{defh}, and so the above sum equals\\ 
$\sum_{x\in L,y\in L} h(x) \pi_x  p_{xy}$. The quantity $\pi_x  p_{xy}$ was shown in \cite{ceff} to equal the `effective conductance' $C_{xy}$ between $x$ and $y$ when the network is finite. Effective conductance is closely related to energy (recall the physical formula $E=I^2 R = I^2 /C$), and we are going to exploit this fact using an argument similar to the proof of \Lr{bodeg}. We cannot directly apply the results of \cite{ceff} as the graph is infinite in our case, however, there is an easy way around this: we can consider an increasing sequence of finite subgraphs $H_1 \subseteq H_2 \subseteq \ldots \subseteq G$ \st\ $\bigcup H_n= G$, apply the results of \cite{ceff} to $H_n$ and take a limit.

To make this more precise, let $p^n_{xy}$ be the probability that \rw\ on $H_n$ from a vertex $x\in L\cap H_n$ exits $(L \cup R) \cap H_n$ at $y\in R\cap H_n$, and let $C^n_{xy}:=\pi_x p^n_{xy}$. It is a well known property of electrical networks, called Reyleigh's monotonicity law, that $C^n_{xy}$ is monotone increasing with $n$ \cite{LyonsBook}. Thus $\lim p^n_{xy} = \lim C^n_{xy}/\pi_x $ exists.
It is now not hard to prove that $p_{xy}\leq \lim p^n_{xy}$ using  coupling and elementary probabilistic arguments. Thus, to prove \eqref{Gp} it suffices to find a uniform upper bound for $\sum_n:= \sum_{x\in L,y\in L}C^n_{xy}$ (where we used the fact that $h$ is bounded). 
It is shown in \cite{ceff} that the above sum $\sum_n$ equals the energy $E_n$ of the harmonic function (or the electric current) on $H_n$ with boundary conditions 1 at $L\cap H_n$ and 0 at $R\cap H_n$. Thus, all we need to do is to prove that the $E_n$ are bounded. 

To achieve this, we will invoke the following well known fact: on a finite network, the harmonic function with given boundary conditions minimises energy among all function satisfying the boundary conditions. This means that it suffices in our case to find arbitrary functions $v^n$ with boundary conditions 1 at $L\cap H_n$ and 0 at $R\cap H_n$ and uniformly bounded energies $E(v^n)$. We will do so by defining a single function $v:V\to \R$ with finite energy $E(v)$ equaling 1 on $L$ and 0 on $R$, and letting $v^n$ be its restriction to $H_n$. Since $E(v^n)\leq E(v)$, the $E(v^n)$ will indeed be bounded.

To define this $v$, let $v(x)=1$ if $x\in L$, let $v(x)=o$ if $x\in R$, and let $v(x)$ be any value in $T(x)$ otherwise, where we think of $T(x)$ as an interval of $\R/Z$ by projecting it to the base of $K$. To check that $E(v)< \infty$, we will compare $E(v)$ with the energy $E(h)$ of the height function $h$ we used in the construction of the tiling. Recall that $E(h)$ equals the area of $K$, which is finite.

Now consider an edge $e=wz$, set $a:= |v(w)-v(z)|$ and recall that the contribution of $e$ to $E(v)$ is $(v(w)-v(z))^2=a^2$. Note that as $e$ joins $w$ to $z$, the construction of the tiling implies that $T(w),T(z)$ meet when projected vertically. Combined with our choice of $v$, this implies that at least one of $T(w),T(z)$ has length at least $a/2$; assume \obda\ this is $T(z)$. Now since $d(z)$ is bounded by $D$, \ti\ at least one edge $f=f(e)$ incident with $z$ \st\ the corresponding square $R_f$ in the tiling has side length at least $a/2D$. This means that the  area of $R_f$ is  at least $a^2/(2D)^2$, which is a constant times the contribution $a^2$ of $e$ to $E(v)$. Now note that each edge $f$ of \g can be considered as $f(e)$ for a bounded number of edges $e$, namely those sharing an endvertex with $f$. This easily implies that $E(v)$ is bounded by a constant times $E(h)$, and so  $E(v)$ is finite since $E(h)$ is. 

\medskip
This proves our claim that the expected number of times that our \rw\ alternates between $L$ and $R$ is finite. Applying this to a countable family of pairs $L_i,R_i$ \st\ the corresponding intervals generate the topology of $\cc$, and combinining this with the almost sure convergence of $h$ along a \rw\ trajectory, immediately yields

\begin{theorem}\label{thconv}
For \rw\ $Z^n$ on \G, the intervals $T(Z^n)$ almost surely converge to a point in $\cc$.
\end{theorem}
\comment{
	\begin{lemma}\label{mer1}
Let $M(x),M(y), x\neq y$ be two meridians, and let $R,L$ be the corresponding meridian-rays. Then $C(L,R)<\infty$.
\end{lemma}
\begin{proof}
By \Lr{} it suffices to find a function $u: V(G^{LR}) \to \R$ with $E(u)<\infty$. We claim that if we let, \fe\ $x\in V(G^{LR})$, $u(x)$ be any number in $T(x)$, then $E(u)<\infty$ holds. To prove this we are going to use the fact that $E(h)<\infty$; indeed, $E(h)= \sum_{xy\in E(G)} (h(x)-h(y))^2$ equals by definition the area of our cylinder. 
\end{proof}
}

This allows us to view $\cc$ as a $G$-boundary as defined in \Sr{secSh}: define $f: \cw \to \cc$ as follows. For a walk $W$ on \g such that $T(W_n)$ converges to a point $p\in\cc$, we let $f(W)=p$; otherwise we let $f(W)$ be a fixed arbitrary point of $\cc$ (such walks form a null-set, so it does not really matter). It is easy to check that $f$ is measurable and shift-invariant. Naturally, we define the measures $\nu_z$ on $\cc$ by $\nu_z(X):= \mu_z(f^{-1}(X))$, making sure that $f$ is measure preserving. 

Combining \Tr{thconv} with \Lr{wn}, it is an easy exercise to deduce the following fact, already proved in \cite{BeSchrST}, that will be useful later

\begin{corollary} \label{Lebesgue}
$\nu_o$ is equal to Lebesgue measure on $\cc$.
\end{corollary}

\section{Faithfulness of the boundary to the \shf s} \label{hard}

In this section we prove that the boundary $\cc$ of the tiling constructed above is faithful to every \shf, which we plug into \Tr{thPB} to complete the proof that $\cc$ is a realisation of the \PBv\ of our planar graph \G\ (\Tr{main}).
 
Let $s$ be such a function, fixed throughout this section.

\subsection{Some basic \rw\ lemmas}

We first collect some basic general lemmas on \rw s that will be useful later. The results of this section hold for general Markov chains, but we will only formulate them for \rw. 

Let $\ca$ be a tail event of our \rw, i.e.\ an event not depending on the first $n$ steps \fe\ $n$.
(The only kind of event we will later consider is the event $1^s$ that $s(Z_n)$ converges to 1, where $s$ is our fixed sharp harmonic function.)

For $r\in (0,1/2]$, let 
\labtequstar{
$A_r:= \{v\in V \mid \PrI{\ca} > 1-r\}$ and\\ 
$Z_r:= \{v\in V \mid \PrI{\ca} < r\}$.
}
Note that $A_r \cap Z_r=\emptyset$ \fe\ such $r$.

By \Cr{hprob}, if we let $\ca:=\cx $ then we have $A_r= \{v\in V \mid s(v)> 1-r\}$ and $Z_r= \{v\in V \mid s(v)< r\}$.

\begin{lemma}\label{Ade}
\Fe\ $\eps,\del\in (0,1/2]$, and every $v\in A_\eps$, we have $\PrII{v}{\text{visit $V \sm A_\del$}} < \eps/\del$. Similarly, \fe\ $v\in Z_\eps$, we have $\PrII{v}{\text{visit $V \sm Z_\del$}} < \eps/\del$.
\end{lemma}
\begin{proof}
Start a \rw\ $(Z_n)$ at $v$, and consider a stopping time $\tau$ at the first visit to $V \sm A_\del$. If $\tau$ is finite, let $z=Z_\tau$. Since $z\not\in A_\del$, we have $ \PrI{\cx^c_z} \geq \del$ by the definition of $A_\del$. Thus, subject to visiting $V \sm A_\del$, the event $ \ca$ fails with probability at least $\del$ since it is a tail event. But $ \ca$ fails with probability less than $\eps$ because $v\in A_\eps$, and so $\PrII{v}{\text{visit $V \sm A_\del$}} < \eps/\del$ as claimed. 

The second assertion follows by the same arguments applied to the complement of $\ca$.
\end{proof}

\note{don't need:
\begin{lemma}\label{uppaths}
For every vertex $v\in A_r$ \ti\ an infinite path starting at $v$ and contained in $A_r$.
\end{lemma}
\begin{proof}
It suffices to show that every component of $A_r$ is infinite, because by K\"onig's lemma \cite{diestelBook05} every connected, infinite, \lfg\ contains an infinite path. 

So suppose, to the contrary, that there is a finite component $K$ in $A_r$. By the maximum principle for harmonic functions, the maximum value of $s(x)$ among all vertices $x$ in $K\cup \partial K$ is attained by some vertex $z$ on the  boundary $\partial K$ of $K$. But then $z$ must lie in $A_r$ since $s(z)> 1-r$ by the definitions, a contradiction.
\end{proof}
}

\begin{corollary}\label{path}
If \rw\ from $v\in A_\eps$ (respectively, $v\in Z_\eps$) visits a set $W\subset V$ with probability at least $\kappa$, then \ti\ a \pth{v}{W}\ all vertices of which lie in $A_{\eps/\kappa}$ (resp.\ $Z_{\eps/\kappa}$).
\end{corollary}
\begin{proof}
Apply \Lr{Ade} with $\del= \eps/\kappa$.
\end{proof}

\begin{lemma}\label{noalter}
\Fe\ $r\in (0,1/2)$, the probability that \rw\ alternates $k$ times between $A_r$ and $Z_r$ is less than $(2r)^k$.
\end{lemma}
\begin{proof}
Applying \Lr{Ade} for $\eps=r$ and $\del=1/2$, we deduce that \fe\ $v\in A_r$, $\PrII{v}{\text{visit $V \sm A_{1/2}$}}<2r$. Since $r<1/2$ we have $Z_r\subseteq V \sm A_{1/2}$, and so $\PrII{v}{\text{visit $Z_r$}}<2r$. Similarly, if $v\in Z_r$, then $\PrII{v}{\text{visit $A_r$}}<2r$.


Using this, and the Markov property of \rw, we can prove by induction that the probability for \rw\ from any vertex to alternate $k$ times between $A_r$ and $Z_r$ is less than $(2r)^{k}$ as claimed. 
\end{proof}

\begin{lemma}\label{nocros}
Let $x,y\in B_m$ and $\eps,\del\leq 1/2$. There is no \pth{x}{y} in $A_\del 
\cap G- G_m$ separating a vertex $z\in B_m\cap Z_\eps$ from $C$.
\end{lemma}
\begin{proof}
If \ti\ such a path, then \rw\ from $z$ in $G- G_m$ would have to visit it with probability 1. But by the Markov property of $Z_n$, this would imply that $z\in A_\del$, contradicting  $z\in Z_\eps$ since $A_\del \cap Z_\eps= \emptyset$.
\end{proof}

\subsection{The main lemma}

Let $s$ be a sharp harmonic function of \G. \Fe\ $n\in N$, we let
\begin{align*}
F_n:= \{b\in B_n \mid s(b)>1/2\} \\
F'_n:= \{b\in B_n \mid s(b)<1/2\}.
\end{align*}

The following lemma is the main ingredient of the proof that $\cc$ is a realisation of the Poisson boundary.
Let $\clos{F_n}$ denote 
the projection of $T(F_n)$ to the base $\cc$ of $K$.

\begin{lemma}\label{lsydi}
$\lim_{m,n} w(\prj{F_n} \sydi \prj{F_m})=0$.
\end{lemma}
\begin{proof}
%

Suppose to the contrary \ti\ $\lambda>0$ \st\ for arbitrarily large $m$ \ti\ an arbitrarily large $n>m$ with $w(\prj{F_n} \sydi \prj{F_m})>2\lambda$.

Since $s$ is sharp, the weak convergence of \eqref{weakly} ---and the remark after it--- easily implies that \ta\ sequences \seq{\eps},\seq{\zeta} \st\ $\lim \eps_i=0$, $\lim \zeta_i=0$ and, defining $X_i$ by
\labtequ{defX}{$X_i:= \{b\in B_i \mid s(b)>1-\eps_i\}$,}
we have $w(F_i \sydi X_i)< \zeta_i$. Note that $X_i = F_i \cap A_{\eps_i}$ by \Cr{hprob}, where $A_r$ is defined as in the previous section for $\ca:=1^s$. We may assume that both \seq{\eps},\seq{\zeta} are decreasing. Fix such sequences, and define $Y_i$ similarly to $X_i$ by $Y_i:= \{b\in B_i \mid s(b)<\eps_i\}$. Assume that $\eps_1<1/2$, to ensure that $X_i,Y_i$ are always disjoint and that $X_i\subseteq F_i$ and $Y_i\subseteq F'_i$. We may also assume that $w(F'_i \sydi Y_i)< \zeta_i$.


Let $\chi:= \lim \mu^i_o(F_i)= \lim \mu^i_o(X_i)$, which exists by \eqref{weakly} and \eqref{harmB}. Note that $1-\chi=\lim \mu^i_o(X^c_i) = \lim \mu^i_o(Y_i)$. In other words, setting $W_i:= B_i \sm (X_i \cup Y_i)$, we have
\labtequ{W}{$\mu^i_o(W_i)$ converges to 0.}
%

\note{$X_n$ is green, $Y_n$ is red, the rest white.}

\comment{
    {\bf Observation:}
    There is no finite green or red component.\\
    Proof: greenness is a harmonic function. If there is a finite component, consider its union with its boundary. By the maximum principle the maximum is         achieved on the boundary, and so the boundary contains a green vertex, a     contradiction.
}


Given \lam\ as above, we can choose $m\in\N$ \leth\ $\eps_m,\zeta_m$ are very small compared to \lam. We will later be able to specify how much smaller we want them to be. 
In order to avoid tedious delta-epsilontics, we
introduce the notation $\Psi \simeq \Omega$, where $\Psi,\Omega$ are functions of $m$, or quantities bounded by functions of $m$, to express the assertion that $|\Psi - \Omega|$ is bounded from above by a function of $m$ that converges to 0 as $m$ goes to infinity.
Similarly, the notation $\Psi \gtrsim \Omega$ means that $\Omega- \Psi$ is bounded from above by such a function if it is positive.
We will use this notation as a means of simplification, the point being that when we introduce an error that can be bounded by some function of $\eps_m,\zeta_m$, we do not have to write down that function explicitly to know that the error is small, if we know that we can make the error as small as we wish by choosing $m$ accordingly large.

Note that $\eps_m,\zeta_m \simeq 0$ by their definition. This implies, for example, that $\eps_m/\lam \simeq 0$ since \lam\ is fixed, an observation we will use later. As another example, we have $\mu^m_o(X_m) \simeq \chi$ by the definition of
$\chi$. Similarly, we have $ \mu^m_o(W_m) \simeq 0$ by \eqref{W}, and $ \mu^m_o(Y_m) \simeq 1-\chi$. Even more, we have
\labtequ{simeq0}{$\sup_{n\geq m} \mu^n_o(X_n) \simeq \chi$,\\
$\sup_{n\geq m} \mu^n_o(W_n) \simeq 0$, and \\
$\sup_{n\geq m} \mu^n_o(Y_n) \simeq 1-\chi$.}

By \Lr{wn} we have $\mu^m_o(X_m)= \sum_{x\in X_m} w(x)$ and $\mu^n_o(X_n)= \sum_{x\in X_n} w(x)$. It thus follows from our assumption $w(\prj{X_n} \sydi \prj{X_m}) \simeq w(\prj{F_n} \sydi \prj{F_m}) >2\lambda$ and \eqref{simeq0} that the proportion of the $m\to n$ projection of $T(X_m)$ which is not in $T(X_n)$ is $\gtrsim\lam$; to make this more precise, we denote by $x\downarrow_n$, where $x$ is an element or a subset of $X_m$, the set of vertices $y\in B_n$ \st\ $T(y)$ and $T(x)$ meet when vertically projected to $\cc$.
Then we have $ w(X_m)- \sum_{y\in X_m\downarrow_n} w(y)\gtrsim \lam$, 
where we used the fact that vertex widths converge to 0 as we move down the tiling \eqref{limax}.

By \eqref{W} we have $\mu^n_o(W_n)\simeq 0$, and so $\mu^n_o(X_m\downarrow_n)\simeq \mu^n_o(X_m\downarrow_n \cap X_n) + \mu^n_o(X_m\downarrow_n \cap Y_n)$. The above fact now implies that a $\gtrsim\lam$ proportion of the $m\to n$ projection of $T(X_m)$ is in $T(Y_n)$, i.e.\ %
\labtequ{grprr}{$\mu^n_o(X_m\downarrow_n \cap Y_n) \gtrsim \lam$.}

We will now use \Lr{Ade} to show that for most vertices $b\in B_m$ \ti\ a \pth{b}{B_n} in $G^m_n$ along which the values of $s$ are close to $s(b)$. To make this more precise, define $A_r,Z_r$ as in the previous section, putting $\ca= 1^s$. Then \Lr{Ade} and the Markov property implies that, for any $\del<1/2$,
\labtequc{AdeC}{$\mu_{o}{\{\text{visit $A_{\eps_m}$ and then $V\sm A_\del$}\}}\leq \mu_{o}{\{\text{visit }A_{\eps_m}\}} \frac{\eps_m}{\del} \leq \frac{\eps_m}{\del}$.}
Now let $\del= \lam$, and call a \pth{X_m}{B_n}\ in $G^m_n \cap A_\del$ an $A$-barrier, and a \pth{Y_m}{B_n}\ in $G^m_n \cap Z_\del$ a $Z$-barrier. Recall that $X_m= F_m \cap A_{\eps_m}$ and $Y_m= F'_m \cap Z_{\eps_m}$. Call a vertex \defi{good} if it is in $X_m$ and is incident with an $A$-barrier or it is in $Y_m$ and incident with a $Z$-barrier; let $W'_m\supseteq W_m$ be the set of all vertices of $B_m$ that are not good. We claim that
\labtequ{good}{$\sum_{x\in W'_m} w(x) \simeq 0$.}
%

Indeed, recall that, by \Lr{wnlast}, $w$ coincides with the distribution of the last visit of our \rw\ to $B_m$. Now subject to this last visit being in $W'_m \cap X_m$, \rw\ will visit $V\sm A_\del$ with probability 1 because \rw\ eventually reaches $B_n$ but \ti\ no \pth{(W'_m \cap X_m)}{B_n}\ in $G^m_n \cap A_\del$, because such a path would be an $A$-barrier, from which we conclude $\mu_o{\{\text{visit $W'_m \cap X_m$ and then $V\sm A_\del$}\}}\geq \sum_{x\in W'_m \cap X_m} w(x)$. On the other hand, we have $\mu_o{\{\text{visit $X_m$ and then $V\sm A_\del$}\}}\leq
\mu_o{\{\text{visit $A_{\eps_m}$ and then $V\sm A_\del$}\}}$, but the latter probability is bounded from above by $\eps_m/\del$ by \eqref{AdeC}. Putting the two inequalities together we obtain $\sum_{x\in W'_m \cap X_m} w(x)\leq \eps_m/\del \simeq 0$. Similarly, we can prove that $\sum_{x\in W'_m \cap Y_m} w(x) \simeq 0$. As $W'_m = (W'_m \cap X_m) \cup (W'_m \cap Y_m) \cup W_m$, our claim follows.

\medskip
We are now going to partition $B_m$ into classes, called \defi{clusters}, in such a way that every cluster is incident with a barrier, and vertices in the same cluster have similar $s$-values except for a subset of small width. For this, define first a \defi{pre-cluster} to be either a maximal subset of $X_m\sm W'_m$ spanning an interval of $L_m$ in the tiling, or a maximal subset of $Y_m\sm W'_m$ spanning an interval of $B_m$. 

Note that every pre-cluster of the first type is, by definition, incident with an $A$-barrier, and every pre-cluster of the second type is incident with a $Z$-barrier.

Next, we associate each vertex in $W'_m$ (these are precisely the vertices not belonging to any pre-cluster) to the first pre-cluster to their `left' along $L_m$, and define a \defi{quasi-cluster} to be the union of a pre-cluster with the set of vertices associated with it. Note that each quasi-cluster is incident with at least one $A$-barrier or $Z$-barrier, but not with both. This allows us to classify quasi-clusters into two types, called type $A$ and type $Z$ accordingly. Finally, we define a \defi{cluster} to be the union of a maximal set of quasi-clusters that spans an interval of $L_m$ and is not incident with both an $A$-barrier and a $Z$-barrier. 

Note that clusters can be still  classified into types as above, as each cluster is incident with either an $A$-barrier or a $Z$-barrier. Let $\Xi_m$ denote the set of clusters of type $A$ and  $\Upsilon_m$ the set of clusters of type $Z$. By construction, no two clusters of the same type are adjacent as intervals of $L_m$.

A cluster of type $A$ may in principle have more of its width in $Y_m$ than $X_m$, but such clusters are rare when measured by $w$. For a typical cluster, its type agrees with the kind of most of its width:
\labtequ{typecol}{$w(\bigcup \Xi_m \sm X_m) + w(\bigcup \Upsilon_m \sm Y_m) \simeq 0$,}
because $w(W'_m) \simeq 0$ and $(\bigcup \Xi_m \sm X_m) \cup (\bigcup \Upsilon_m \sm Y_m) \subseteq W'_m$.

Combined with \eqref{grprr}, this implies that the proportion of the $m\to n$ projection of the clusters of type $A$ that is in $Y_n$ is at least  $\gtrsim\lam$, i.e.\
\labtequ{grprr2}{$\sum_{x\in \bigcup \Xi_m} \mu^n_o(x\downarrow_n \cap Y_n) \gtrsim \lam$.}

Let
$$\Xi'_m:= \{ c\in \Xi_m \mid \frac{w(c\downarrow_n \cap Y_n)}{w(c)}\geq \lam/2\},$$
and note that $\Xi'_m$ consists of those clusters of type $A$ that have a relatively large contribution to \eqref{grprr2}.
We claim that $\Xi'_m$ is a relatively large set, namely, $w(\bigcup \Xi'_m) \gtrsim \frac{\lam(1-\chi/2)}{1-\lam/2}$. For this, let $P= w(\bigcup \Xi_m\downarrow_n \cap Y_n)$,
and recall that $P \gtrsim \lam$ by \eqref{grprr2}. Note that $w(\bigcup \Xi'_m) + \lam/2 (\chi - w(\Xi'_m))\gtrsim P$ by the definition of $\Xi'_m$.
Since $P \gtrsim \lam$, an easy calculation proves our claim $w(\Xi'_m)\gtrsim \frac{\lam(1-\chi/2)}{1-\lam/2}$. Since  $\chi\leq 1$, this implies that
\labtequ{wXp}{$w(\Xi'_m)\gtrsim\lam/2$.}
%


%
%

Since the proportion of the $m\to n$ projection of the clusters of type $A$ that is in $Y_n$ is at least  $\gtrsim\lam$ \eqref{grprr2}, the results of \Sr{secPar} (\Lrs{wnlast} and \ref{wn}) imply that with positive probability our \rw\ has its last visit to $B_m$ in some cluster $c$ of type $A$ and then diverts to hit $B_n$ outside $c\downarrow_n$. Let us make this more precise. We say that \rw\ from $o$ \defi{goes astray}, if its final $m$-subwalk (this terminology was introduced in the proof of \Lr{stretch}) starts at a cluster $c$ and does not hit $B_n$ in $c\downarrow_n$. We claim that
\labtequ{astray}{Conditioning \rw\ from $o$ to have its last visit to $B_m$ in $\Xi'_m$, it goes astray with probability $\gtrsim \lam/2$.}
To prove this, recall that the distribution of the last vertex on $B_m$ visited (which coincides with the first vertex of the  final $m$-subwalk) coincides with $w$ by \Lr{wnlast}, and that the hitting distribution to $B_n$ coincides with $w$ (\Lr{wn}). Thus the definition of $\Xi'_m$ implies that if our claim \eqref{astray} is wrong, then the probability that the final $m$-subwalk starts in $\Xi'_m$ and exits in $Y_m$ is not $\simeq 0$. This probability however is bounded from above by $\mu_o{\{\text{visit $X_m$ and then $Y_m$}\}}$, which by \eqref{AdeC} is at most $\frac{\eps_m}{1/2} \simeq 0$, which leads to a contradiction. This proves  \eqref{astray}.

\medskip
The fact that \rw\ goes astray with positive probability combined with \Lr{Ade} implies the existence of certain \pths{L_m}{L_n}\ with vertices in $A_\del$ that our \rw\ can follow, which can be proved by arguments similar to those that we used above to prove the existence of barriers. To make this more precise, let $\Xi''_m$ be the set of those clusters $c$ in $\Xi'_m$ \st\ \ti\ a path $\Puz$  in the strip $G_n^m$ with vertices in $A_{\del}$ from $c$ to the set $B_n\sm c\downarrow_n$ of vertices not lying below $c$ in the tiling (\fig{puz}). We claim that most of $\Xi'_m$ is in $\Xi''_m$:
\showFig{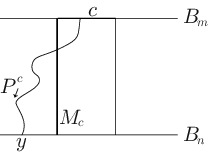}{The path \Puz.}
\labtequ{Puz}{$w(\bigcup \Xi'_m \sm \bigcup \Xi''_m)\simeq 0$.}
To see this, note that subject to going astray having started the final $m$-subwalk in $\bigcup \Xi'_m \sm \bigcup \Xi''_m$, our \rw\ visits $V\sm A_{\del}$ with probability 1, and so this happens with probability $\simeq 0$ by \eqref{AdeC}.

Note that every such path $\Puz$ is an $A$-barrier by definition. But what is special about it, is that it must cross one of the meridians of $c$, in fact the part of the meridian contained in the strip $G^m_n$, which part we denote by $M_c$. A more careful examination of the proof of the existence of $\Puz$ shows that we can strengthen \eqref{astray} to demand that \rw\ goes astray crossing $M_c$.




\medskip
Since $w(\Xi''_m)\gtrsim\lam/2$ by \eqref{wXp} and \eqref{Puz}, we have just proved  that the final $m$-subwalk of \rw\ from $o$ crosses an $M_c$ as above with probability $\gtrsim \lam^2/4$.

These crossings account for a flow of particles from one side of these meridians to the other. By \Lr{mer}, this flow has to be compensated for by other particles, crossing the meridian in the other direction. Such particles have to cross not only $M_c$, but also $\Puz$ by a topological argument; in particular, they have to visit $A_{\del}$. To see that such particles do indeed have to cross $\Puz$, suppose for a moment that \fig{puz} is not showing part of a cylinder, but part of a rectangular strip, i.e. we have cut the cylinder along a perpendicular line not visible in the figure. Let $P$ be a path that starts in $B_m$, finishes in $B_n$, and crosses $M_c$ from left to right. For our purposes, it is enough to consider such a $P$ whose width coordinates grow monotonely, for if a particle crosses $M_c$ in one direction and then back in the other direction then it has no compensating contribution to the aforementioned flow. Thus $P$ starts on $B_m$ on the left of $M_c$, and finishes on $B_n$ on the right of $M_c$. But then $\Puz$  separates the two endpoints of $P$, and so $P$ has to cross $\Puz$. 

Now since we are on a cylinder rather than a strip, the words `left,right' make less sense. However, for this argument it makes little difference; for example, we can consider a covering of the cylinder $K$ by an infinite strip, lift $P, M_c$ and $\Puz$ to the covering, and then apply the above argument.

Recall that only the initial and final $m$-subwalk can have a non-zero contribution to the expected number of traversals  of an edge by \Lr{ww}, and so only the final $m$-subwalk can have a compensating contribution for the above flow since $M_c$ lies below $L_m$. We can bound the probability with which the particle starts its final $m$-subwalk in $Y_m$ and displays this behaviour, i.e.\ crosses an $M_c$ and hence visits $A_{\del}$, applying \eqref{AdeC} with the roles of $A$ and $Z$ reversed: this probability is at most $\eps_m/\del \simeq 0$.

Note however, that although  a  particle visiting $Y_m$ crosses an $M_c$ with low probability, it might be that when it does so it crosses the $M_c$ of several clusters $c$, and so its expected contribution in compensating the above flow becomes considerable. But we can use \Lr{noalter} to bring this possibility under control as follows.
Again, only the   final $m$-subwalk is relevant. Now suppose that the final $m$-subwalk of \rw\ from $o$ crosses the $M_c$ of several $c\in \Xi''_m$, say $k$ of them. Then, as no two clusters of type $A$ can be adjacent as intervals of $L_m$, and as every cluster has a barrier of its own type, such a subwalk must, by a geometric argument, visit at least $k$ $A$-barriers and $k-1$ $Z$-barriers alternatingly. Thus, such a subwalk alternates between  $A_{\del}$ and $Z_{\del}$ $k$ times. The probability of this event though can be bounded using \Lr{noalter}: using an argument similar to the one we used to derive \eqref{AdeC} from \Lr{Ade}, we see that the expected number of alternating visits between $A_{\del}$ and $Z_{\del}$ is less than $2$.
Combining this fact with the above bound for the probability of even visiting a single $A$-barrier  from $Y_m$, and using the Markov property of \rw, we deduce that the expected number of $A$-barriers crossed by final $m$-subwalks starting in $Y_m$, which equals the compensating contribution of $Y_m$ to the above flow, is $\simeq 0$.

Now consider the case where the particle starts it final $m$-subwalk at $x\in X_m$. If $x$ happens to be in $\Xi''_m$, then it cannot cross the $A$-barrier of its own cluster $c$ in a compensating manner, i.e.\ coming into the vertical strip under $c$ from outside, unless it travels all the way around the cylinder. 
In any case, in order for such a subwalk to have a compensating contribution, it has to cross the $A$-barrier of some other cluster $z\in \Xi''_m$. But then it would have to visit the $Z$-barriers of all the cluster of type $Z$ lying between $c$ and $z$, and there is at least one such cluster by construction. 
Using the same arguments as in the previous case, we deduce that the expected compensating contribution of $X_m$ to the above flow is also less than $\simeq 0$.

Finally, \rw\ starts its final $m$-subwalk in $W_m$ with probability $\simeq 0$, and again its  compensating contribution is $\simeq 0$ too by similar arguments.

This proves that our alleged flow crossing the $M_c$ cannot be compensated for in accordance with \Lr{mer}, yielding a contradiction that completes our proof.
\end{proof}

Note that we can replace $F_m,F_n$ with $X_m,X_n$, defined as above \eqref{defX}, in the assertion of \Lr{lsydi} because $w(F_i \sydi X_i)< \zeta_i$ and $\lim \zeta_i=0$, which means that
 \labtequ{wFX}{$\lim w(F_i \sydi X_i) =0$.}

\medskip

An \defi{$(m,n)$-impurity} is a vertex in $F_n$ that 
has an $m$-predecessor in $F'_m$, or a vertex in $F'_n$ that 
has an $m$-predecessor in $F_m$, where an $m$-predecessor of $y$ is a vertex $x\in B_m$ \st\ $y\in x\downarrow_n$. Let $M^m_n\subseteq B_n$ be the set of $m,n$-impurities. Using this terminology, \Lr{lsydi} can be reformulated as follows

\labtequ{impur}{$\lim_{m,n} w(M^m_n)=0$.}

\subsection{Faithfulness} \label{subfaith}

Using \eqref{impur} we can now show that $\cc$ is faithful to $\cs$.

\begin{lemma} \label{lemFai}
\Fe\ \shf\ $s$ there is a measurable subset $X$ of $\cc$ \st\ $\mu_o(1^s_o \sydi f^{-1}(X))=0$.
\end{lemma}
\begin{proof}
Let $\eps_i= 2^{-i}$ and apply \Lr{lsydi}, or \eqref{impur}, to obtain a sequence $(n_{\eps_i})_\iin$ \st\ \fe\ $m,n\geq n_{\eps_i}$ we have $w(M^m_n)<2^{-i}$. Since our choice of the levels $l_n$ was arbitrary, we may assume \obda\ that $n_{\eps_i}=i$, and we will make this assumption to avoid subscripts in our notation. Our choice of $(\eps_i)$ has the effect that $\sum_m:= \sum_{n>m} w(M^m_n) < \infty$ \fe\ $m$, and $\lim_m \sum_m = 0$. In other words, we have
\labtequ{impurvan}{$\lim_m \mu_o(\{\text{exit some $n>m$ at an $(n-1,n)$-impurity}\})=0$.} 
Recall the definition of $F_i$ from the previous section, and let $\clos{F_i}$ denote 
the projection of $T(F_i)$ to the base $\cc$ of $K$. Let $X$ be the set of points eventually in $\overline{F_{i}}$, i.e.\ 
$$X:= \{x\in \cc \mid \text{\ti\ $k$ \st\ $x\in \overline{F_{i}}$ \fe\ $i>k$}\},$$
and define $Y$ similarly replacing $F_i$ by $F'_i$. 

Note that $X=  \liminf_i \overline{F_{i}} := \bigcup_j \bigcap_{i\geq j} \overline{F_{i}}$. 
Thus $X$ is a borel subset of $C$, hence $\nu$-measurable. We will prove that $X$ has the desired property\\ 
$\mu_o(1^s_o \sydi f^{-1}(X))=0$.

Similarly to \eqref{impurvan} it follows from the summability of $(\eps_i)$, and the definition of  $M^m_n$, that $\lim_m \sum_{n>m} w(\clos{F_n} \sydi \clos{F_{n-1}}) =0$. This implies
\labtequ{sydiX}{$\lim_m  w(\clos{F_m} \sydi X) =0$ and $\lim_m  w(\clos{F'_m} \sydi Y) =0$}
because any point $p\in \clos{F_m} \sydi X$ lies in $\clos{F_n} \sydi \clos{F_{n-1}}$ for some $n> m$ by the definition of $X$.

Next, we claim that \ti\ a sequence $(n_i)_\iin$ of levels \st\
\labtequ{switchvan}{$\sum_i \mu_o(\{\text{visit $F_{n_i}$ and then visit $F'_{n_j}$ for some $j>i$} \})<\infty$.} 
For this, recall that visiting $F_{n_i}$ is almost the same event as visiting $X_{n_i}$ by \eqref{wFX} (recall the definition \eqref{defX} of $X_n$); similarly, visiting $F'_{n_j}$ is almost the same event as visiting $Y_{n_j}$, and it is unlikely to visit any $Y_{n_j}$ after having visited $X_{n_i}$ by \Lr{Ade}. We can thus use the above technique of choosing a sequence $(n_i)_\iin$ corresponding to the sequence $\eps_i$ where the $\eps_i$ are used as upper bounds of probabilities of the unlikely events we want to avoid, to achieve \eqref{switchvan}.

Again, we may assume \obda\ that $n_i=i$ as we did above.

Using \eqref{impurvan}, \eqref{switchvan} and the \BCl, we deduce that our \rw\ almost surely exits finitely many $B_n$ at an $(n-1,n)$-impurity and switches between the $F_i$ and the $F'_i$ finitely often. This means that 
almost surely \ti\ some $m$ \st\ its trajectory converges to a point in $\clos{F_m}$ if it has hit $B_m$ in $F_m$ and to a point in $\clos{F'_m}$ if it has hit $B_m$ in $F'_m$, having exited each $B_n, n>m$ at $F_n$ in the former case and at $F'_n$ in the latter. Note that conditioning on the first case we have almost sure occurrence of the event $1^s$, while conditioning on the latter we have $0^s$ almost surely. 

Moreover, the probability with which the limit of the \rw\ trajectory lies in $\clos{F_m}$ but not in $X$ (and similarly with $F'_m$ and $Y$) is at most $w(\clos{F_m} \sydi X)$ because $w$ coincides with the exit measure by \Lr{wn}. Since, clearly, we can choose $m$ arbitrarily large, this probability can be made arbitrarily small by \eqref{sydiX}. Putting these remarks together, we obtain $\mu_o(1^s_o \sydi f^{-1}(X))=0$ as desired.
\end{proof}

The following is essentially a reformulation of \Cr{cortop}.

\begin{corollary}
Let $G$ be a locally finite transient graph and let $\mu_o$ be the distribution of \rw\ from $o\in V$. Let $(\cn,\co)$ be a topological space endowed with  Borel measures $(\nu_z)_{z\in V}$ and a `projection' $\tau: V \to \co$ so that the following are satisfied
\begin{enumerate}
\item there is a Borel-measurable function $\tau^*: \cw \to \cn$ mapping almost every $(Z^n)\in \cw$ to $\lim_n \tau(Z^n)$ ---in particualr, for \rw\ $Z^n$ on $G$, $\tau(Z^n)$ converges in $\cn$ $\mu$-almost surely--- and $\mu_z(\{\tau^*((Z^n)) \in O\}) = \nu_z(O)$.
\item \ti\ a sequence \seq{G} of subgraphs of \g with $\bigcup G_n= G$ \st\ \rw\ visits all boundaries $B_n$ of $G_n$ almost surely (this is automatically satisfied when $G_n$ is finite), and $\mu_o^n(b) = \nu_o \circ \tau(b)$ \fe\ $b\in B_n$, where $\mu_o^n$ denotes the exit distribution of $G_n$ for \rw\ from $o$.
\item \fe\ \shf\ $s$ and every $z$, we have\\ 
$\lim_{m,n} \nu_z(\tau(F_m) \sydi \tau(F_n)) =0$, where $F_i:= \{b\in B_n \mid s(b)>1/2\}$.
\end{enumerate}
Then $\cn$ is a realisation of the Poisson boundary of $\G$.
\end{corollary}
\begin{proof}
Note that $\cn$ endowed with the measures $(\nu_z)_{z\in V}$ satisfies the axioms for being a $G$-boundary.
The proof of \Lr{lemFai} applies almost verbatim, replacing $\cc$ with $\cn$ and $w$ with $\nu$, the role of $\tau$ being played by the `horizontal' position in the tiling. The set $M^m_n$ of $(m,n)$-impurities can be defined in general as $\tau(F_m) \sydi \tau(F_n)$. This yields that $\cn$ is faithful to every \shf, and so by \Tr{thPB} it is a realisation of the Poisson boundary.
\end{proof}

In condition (iii) we can replace $F_m,F_n$ by sets $X_m,X_n$ defined as above \eqref{defX} because of \eqref{wFX}, and it will probably be helpful to do so in applications.

\medskip
Recall that the only cases in which we used the bounded degree assumption on \g were in proving that the widths $w(x)$ converge to zero \eqref{limax} and that the image of \rw\ almost surely converges to $\cc$ (\Tr{thconv}). Thus we can strengthen \Tr{main} as follows
\begin{corollary}\label{unbou}
If \g is a plane, uniquely absorbing graph \st\ \\ $\limsup_n \max_{b\in B_n} \mu_o^n(b)=0$, and the image of \rw\ in the square tiling of \g almost surely converges to the boundary $\cc$, then $\cc$ is a realisation of the Poisson boundary of \G.
\end{corollary}

Here, $B_n$ is the boundary of the $n$th member of a sequence $G_n$ as in our proof, and we applied \Lr{wn} to rewrite \eqref{limax} in terms of the exit distributions $\mu_o^n$.

\comment{ 
	\sectio{Every measurable subset of $C$ equals some $f(\cx_o)$ up to a null-set} \label{secConv}
\labtequ{}{}
In this section we show that \cn\ has the second condition required by \Tr{thPB}, namely the one in the title.

It suffices to prove the assertion for a closed interval $X$ of $C$, for these sets generate the \sig-algebra of $C$ under complementation and countable unions, and the family of \stoco\ sequences is closed under these operations (\Lr{comb}).

So given such $X\subseteq C$, let $X_n:= \{x\in B_n \mid \overline{x} \subset C\}$ be the set of vertices of $B_n$ lying above $X$. Since \rw\ converges to a point of $C$ almost surely, and since it converges to one of the endpoints of $X$ with probability 0, it eventualy stays either inside or outside the strip above $X$ for ever. This means that either it hits almost all $X_n$ or almost none. Thus the \seq{X}\ is a stable sequence, and by \Lr{lstab} \stoco, and the event of converging to a point in $X$ coincides up to a null-set with $f(\cx_o)$.
}

\section{The Northshield circle}
Northshield \cite{NorCir} endowed every plane graph \g satisfying properties \ref{Ni}--\ref{Nl} below with compactifications $|G|_\sim$ and $|G|_{\cong}$, and essentially conjectured that the boundary $\cbg$ of the former, which he showed to be homeomorphic to a circle, is a realisation of the Poisson boundary, while the boundary \cgbg\ of the latter is homeomorphic to the Martin boundary of \G. 
 We will confirm the former conjecture by showing that it is canonically homeomorphic to the boundary of the square tiling of \G.  

Every (plane) graph \g considered by Northshield has the following properties
\begin{enumerate}
\item \label{Ni} \g is \defi{\vapf}, i.e.\ no point of $\R^2$ is an accumulation point of vertices of \G\footnote{Northshield requires instead that every cycle in \g surrounds only finitely many vertices; that this is equivalent to \ref{Ni} as proved in \cite[Lemma~7.1]{thoPla}.};
\item \label{Nii} \g is \defi{non-amenable}, i.e.\ $\inf_K \frac{|\partial K|}{ |K|}>0$, where $K$ ranges over all finite vertex-sets of \G, and its boundary $\partial K$ is the set of vertices outside $K$ sending an edge to $K$;
\item \label{Nl} \g has bounded degrees.
\end{enumerate}

We now repeat the definition of $|G|_\sim$ and $|G|_{\cong}$ from \cite{NorCir}.
The corresponding boundaries are defined as equivalence classes of geodesics in \g with a common starting vertex $o$. Let $\Gamma= \Gamma_o$ denote the set of these geodesics. For any two elements $R,L\in \Gamma$ with finite intersection, $\{R \cup L\}$ divides the plane, and hence \g, into two unbounded open domains by \ref{Ni}, and we will denote by $(R,L)$ and $(L,R)$ the intersection of each of these domains with \G, where we use a fixed orientation of the plane to decide which is which. Note that $(R,L)$ and $(L,R)$ can be empty; this occurs exactly when $R \cup L$ bounds an infinite face of \G. We use  $[R,L]$ to denote the subgraph of \g spanned by $R \cup (R,L) \cup L$.

Given $R,L\in \Gamma$, we write $R \sim L$ whenever $(R,L)$ or $(L,R)$ is a recurrent graph. We write $R \cong L$ whenever there exists a \defi{transient} 1-way infinite path $Q$ such that both $R \cap Q$ and $L \cap Q$ contain infinitely many vertices; here, we call $Q$ \defi{transient}, if \rw\ on \g \as\ visits $Q$ finitely often. We let $\cbg$ denote the set of equivalence classes of $\sim$ and let $\cgbg$ denote the set of equivalence classes of $\cong$. Then we let $|G|_\sim = \g \cup \cbg$ and $|G|_{\cong}= \g \cup \cgbg$. It remains to specify the topologies of these spaces, and this is done in a natural way using the embedding of \g in the plane as follows.

Let $\simeq$ denote any of the relations $\sim$ or $\cong$. Given an equivalence class $[Q] \in \Gamma/ \simeq$, we will, with a slight abuse of notation, write  $[Q] \in (R,L)$ to denote that either (a) $| R \cap L | = \infty$, or (b) $R,L \not\in [Q]$ and \fe\ $W\in [Q]$, a cofinal subgeodesic of $W$ is contained in $(R,L)$. Using this notation, we specify a base for the topology of $|G|_\simeq$ by declaring all sets of the form $(R,L) \sm S$ to be open, where $R,L\in \Gamma$ and $S \subset \G$ is finite.

By the following theorem, \cbg\ is homeomorphic to $\S^1$; we will call it the \defi{Northshield circle} of \G.

\begin{conjecture}[{\cite[CONJECTURE 2.6]{NorCir}}] \label{Ncon}
Let \g be an \vapf\ plane, non-amenable graph with bounded vertex degrees. Then every bounded harmonic function is the integral of a bounded measurable function \wrt\ harmonic measure on the Northshield circle.
\end{conjecture}

\begin{theorem}[{\cite[THEOREMS 2.2, 2.3 and 2.7]{NorCir}}] \label{TN}
Let \g be a \vapf\ plane non-amenable graph with bounded vertex degrees. Then 
\begin{enumerate}
 \item \label{Ncirc} \cbg\ is homeomorphic to a circle;
 \item \label{conv2} \rw\ on \g converges to a point in \cbg\ almost surely, and
 \item \label{Rtrans} for every geodesic $R$ and $v\in V(G)$, we have $\mu_v(\sgl{[R]})=0$.
\end{enumerate}

\end{theorem}

\comment{
	\begin{theorem}[{\cite[Theorem 6.1]{BeSchrST}}] \label{convC}
Let $z_0=o,z_1,\ldots$ be a simple \rw\ on a transient, 1-ended, plane graph \G\  with bounded vertex degrees. Then the images $\tau(z_i)\in T$ converge to a point $x$ in $\ccc(G)$ almost surely, and the distribution of $x$ coinsides with Lebesgue measure on $\ccc$. 
	\end{theorem}
	
\begin{theorem}[{\cite[THEOREM 2.3]{NorCir}}] \label{conv2}
Let \g be a \vapf\ plane non-amenable graph with bounded vertex degrees. Then \rw\ on \g converges to a point in \cbg\ almost surely.
\end{theorem}

\begin{theorem}[{\cite[THEOREM 2.7]{NorCir}}] \label{Rtrans}
Let \g be a \vapf\ plane non-amenable graph with bounded vertex degrees. For every geodesic $R$ and $v\in V(G)$, we have $\mu_v(\sgl{[R]})=0$.
	\end{theorem}
}

\begin{lemma}[{\cite[LEMMA 2.1]{NorCir}}] \label{between}
Let \g be as in \Tr{TN} and $R\not\sim L$ be geodesics in $\Gam$. Then $(R,L)$ contains a geodesic $Q$ with $Q\not\sim R, L$.
\end{lemma}

\begin{lemma}[{\cite[LEMMA 4.2]{NorCir}}] \label{cong}
Let \g be as in \Tr{TN} and $R, L \in \Gam$. If $R\cong L$ then $R\sim L$.
\end{lemma}

\subsection{Proof of Northshield's conjecture}

In this section we prove the following result, which provides a positive answer to \Cnr{Ncon}.
\begin{theorem}\label{thm}
Let \g be an \vapf\ plane, non-amenable graph with bounded vertex degrees. Then the Northshield circle \cbg\  is canonically homeomorphic to the boundary \cc\ of the square tiling of \G. Thus the former is a realisation of the Poisson boundary.
\end{theorem} 
\begin{proof}
To begin with, note that if \g is \vapf\ then it is uniquely absorbing by the definitions. Moreover, it is transient by \Tr{TN}, and so the square tiling and its boundary are well-defined.

We will define a homeomorphism $h: \partial_\sim G \to \ccc$ directly, by showing that the image of any geodesic $R$ of \g in the tiling converges to a point of \ccc\ depending only on the limit of $R$ in $\partial_\sim G$ and not on $R$ itself.

Suppose, to the contrary, that for some geodesic $R$ with edges $e_1, e_2, \ldots$, there are distinct meridians $M,N$ of $T$ each met by $\tau(e_i)$ for infinitely many $i$. Then, since the image of \rw\ $(z_i)$ converges to a point between $M$ and $N$ with positive probability by \Cr{Lebesgue},  planarity implis that $z_i$ must meet $R$ infinitely often with positive probability. This however contradicts \Tr{TN} \ref{Rtrans}.

By a similar argument, if $R,L$ are geodesics \st\ $R \sim L$, then their images under $\tau$ on $T$ converge to the same point of \ccc, for \rw\ converges to a point in \cbg\ by \Tr{TN} \ref{conv2}.

This allows us to define a map $h: \partial_\sim G \to \ccc$ by mapping any $x\in \partial_\sim G$ to the unique point $p$ of \ccc\ such that for every geodesic $x_1,x_2,\ldots$ in $x$, we have $\tau(x_i) \to p$.

To show that $h$ is a homeomorphism, let us first check that it is injective. For this, suppose there are geodesics $R,L$ \st\ $h([R])= h([L])= p\in \ccc$, and so $\tau(R_i) \to p$ and $\tau(L_i) \to p$ by the definition of $h$. Then \Cr{Lebesgue} implies that 
\labtequ{outs}{\rw\ eventually avoids $(R,L)$ \as.}
We would like to show that $R\sim L$. If this is not the case, then using \Lr{between} we can find a geodesics $Q$ in $(R,L)$ and a geodesic $P$ in $(Q,L)$ such that all $R,Q,P,L$ are pairwise non-equivalent. 

We distinguish two cases. Suppose first that \ta\ infinitely many, pairwise disjoint paths $B_1,B_2,\ldots$ joining $R$ to $Q$ in $[R,L]$. Then combining these paths with parts of $R$ and $Q$, we can construct an infinite path $X$ interesecting each of $R,Q$ infinitely often. By \eqref{outs}, this path proves that $R \cong Q$. This however contradicts \Lr{cong}. 

Thus this case cannot occur, and so no such family of paths can exist. This means that \ti\ a finite set $S_1$ of vertices separating $R$ from $Q$ in $[R,L]$ \cite[p. 202]{diestelBook05}. Similarly, \ti\ a finite set $S_2$ of vertices separating $P$ from $L$ in $[R,L]$. This implies that $S_1 \cup S_2$ is a finite set separating $R$ and $L$ from $[Q,P]$ in $[R,L]$; hence  $S_1 \cup S_2$ separates $[Q,P]$ in \G. By our choice of $Q,P$, the subgraph $[Q,P]$ is transient. This easily implies that \rw\ in \g stays eventually in $[Q,P]$ with positive probability, for if it visits a vertex in $[Q,P] \sm (S_1 \cup S_2)$, it might never visit $(S_1 \cup S_2)$ again. But this contradicts \eqref{outs} because $[Q,P]\subseteq [R,L]$.

This proves that $h$ is injective. By its definition, $h$ preserves the natural cyclic ordering induced on the circles \cbg\ and \ccc\ by the embedding of \G, and so $h$ is continuous. 
It easily follows by topological arguments that $h$ is onto, hence a homeomorphism.
\end{proof}

\begin{lemma}\label{simcong}	
Let \g be a \vapf\ plane non-amenable graph with bounded vertex degrees and no infinite faces. Then \fe\  $R,L\in \Gamma$, we have $R \sim L$ \iff\ $R \cong L$, and so \cbg\ coincides with \cgbg.
\end{lemma}
\begin{proof}
By \Lr{cong} $R \cong L$ implies $R \sim L$, so it remains to show that if $R \sim L$ then $R \cong L$.
For this, note that since \g has no infinite faces, $[R,L]$ contains an infinite set of pairwise disjoint $R$-$L$~paths. Combining such paths with parts of $R$ and $L$ we hence obtain a path $X$ in $[R,L]$ visiting each of $R,L$ infinitely often. We will show that $X$ is transient, which means that $R \cong L$.

We distinguish three cases. Suppose first that with positive probability, \rw\ on \g stays eventually inside $[R,L]$. Then the subgraph $[R,L]$ is transient, contradicting our assumption that $R \sim L$. 

The second case is where \rw\ visits both $[R,L]$ and its complement infinitely often with positive probability. This implies, using planarity,  that at least one of $R,L$ is visited infinitely often with positive probability. But this contradicts \Tr{TN} \ref{Rtrans}.

The only remaining case is where \rw\ visits $[R,L]$ only finitely often \as. But then our path $X$ is transient indeed, proving that $R \cong L$ as desired.
\end{proof}

\subsection{Hyperbolic graphs} \label{secHyp}

In this section we prove \Cr{cor} and disprove another conjecture of \cite{NorCir} saying that graphs as in \Tr{TN} are necessarily hyperbolic. We assume that the reader is familiar with the basic notions about hyperbolic graphs and their boundaries; otherwise \cite{GhHaSur,Short} can be consulted.

One of the homeomorphisms in \Cr{cor} was already known: that the Martin boundary coincides with the hyperbolic boundary in a consequence of

\begin{theorem}[{\cite{AncNeg,Anc}}] \label{ancona}
Let \g be an infinite, Gromov-hyperbolic,  non-amenable, graph (bounded degree?). Then the hyperbolic boundary of \g is canonically homeomorphic to the Martin boundary. 
\end{theorem}

Moreover, it is proved in \cite[LEMMA 5.1]{NorCir} that if \g is hyperbolic in addition to satisfying the above conditions, then the hyperbolic boundary \hybg\ coincides with \cgbg\ as a set; more precisely, two geodesics are equivalent \wrt\ to $\cong$ \iff\ they converge to the same point of \hybg. 
It follows that the corresponding topologies also coincide in that case:

\begin{lemma}\label{hypcong}	
Let \g be a hyperbolic, \vapf\ plane, non-amenable graph with bounded vertex degrees. Then \hybg\ is canonically homeomorphic to \cgbg.
\end{lemma}
\begin{proof}
As already mentioned, there is a canonical bijection between \hybg\ and \cgbg, so it just remains to show that their topologies coincide. The topology of $\cgbg$ is, by definition, generated by the `intervals' of the form $\cls{(R,L)} \cap \cgbg, R,L \in \Gamma$. We will show that the same sets generate the topology of \hybg. 

For this, suppose that the geodesics $L,M,R\in \Gamma$ are non-equivalent and that $M\in [L,R]$. 
Let $\pi_{LM},  \pi_{MR}$ and $\pi_{RL}$ be two-way infinite geodesics joining the corresponding limit points in \hybg. We may assume without loss of generality that both $\pi_{LM}$  and $ \pi_{MR}$ lie in the region of the plane bounded by $\pi_{RL}$ that does not contain $o$, for otherwise we can replace any subpath crossing to the other region by a subpath of $\pi_{RL}$ of the same length. For $x,y\in \hybg$, define the 
\defi{Gromov product} $x \hat{} y$ as the shortest length of a path from $o$ to any $x$-$y$~geodesic. It is known \cite{Short} that for every hyperbolic graph \g 
the topology of \hybg\ can be generated by the sets of the form $O_{x,r}:= \{y\in \hybg \mid x \hat{} y > r \}$, where $x$ ranges in all of \hybg\ and $r$ in $\R_{>0}$. Now since $\pi_{LM}$  and $ \pi_{MR}$ lie in the region of the plane bounded by $\pi_{RL}$ that does not contain $o$, it follows that $d_h(\partial L, \partial M) , d_h(\partial M, \partial R)\leq  d_h(\partial L, \partial R)$. This easily implies our claim that the intervals $\cls{(R,L)} \cap \hybg, R,L \in \Gamma$ generate the topology of \hybg.
\end{proof}

Combining \Lrs{hypcong} and \ref{simcong} with Theorems~\ref{ancona} and \ref{thm} yields \Cr{cor}. For this, we use the fact that an 1-ended plane graph is always \vapf\ \cite{ThomassenRichter}. 

We now adress the question of whether all conditions in the assertion of \Cr{cor} and \Tr{thm} are needed. 

To see that having no infinite faces is essential in \Cr{cor}, let \g be a graph satisfying all conditions, pick a 2-way infinite geodesic $R$, and let $G'$ be obtained from $G$ by deleting one of the two sectors into which $R$ divides $G$. Then the hyperbolic and Martin boundary becomes an arc, while the Northshield circle and the square tiling boundary remain circles.

Northshield conjectured that hyperbolicity is implied by the other conditions in \Cr{cor}:

\begin{conjecture}[{\cite[CONJECTURE 5.3]{NorCir}}] \label{Ncon2}
Let \g be a plane, \vapf, non-amenable graph with bounded vertex degrees. Then \g is hyperbolic.
\end{conjecture}

\fig{nohyp} shows a 1-ended counterexample without infinite faces. It is constructed recursively as follows. Start with a triangle $L_0$ embedded in $\R^2$. Draw a new edge incident to each vertex of $L_0$ inside the unbounded region of $\R^2$, and join the new vertices with a cycle $L_1$. Then perform infinitely many steps of the following type: at step $i=1,2,\ldots$, subdivide each edge of $L_i$ by placing $i$ new vertices on it, draw a new edge incident to each vertex (new or old) of $L_i$ inside the unbounded region of $\R^2$, and join their other ends with a cycle $L_{i+1}$.

\showFig{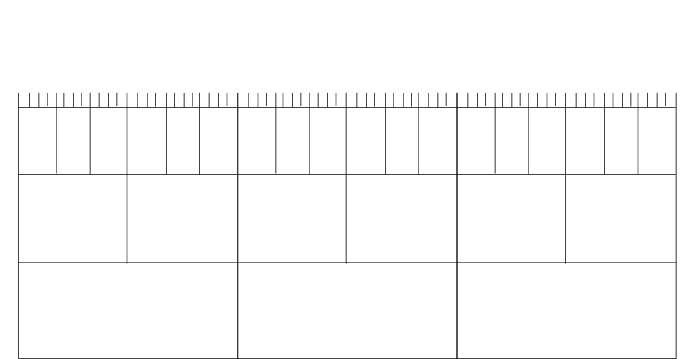}{A counterexample to \Cnr{Ncon2}. The righmost and leftmost geodesics are to be thought as being identified. The cycles $L_i$ are drawn horizontally.}

We claim that $\inf_K \frac{|\partial K|}{ |K|}\geq 1$ for this graph, which means that it is non-amenable. This can be proved by induction on the number of vertices in a smallest hypothetical finite set $K$ \st\ $\frac{|\partial K|}{ |K|}< 1$: consider the highest level $L_n$ of the graph containing a vertex of $K$. If $L_n \subset K$, then letting $K' = K \sm L_n$ we easily have $\frac{|\partial K'|}{ |K'|}<\frac{|\partial K|}{ |K|}$, contradicting the minimality of $K$. Otherwise, there is a vertex $x\in L_n\cap K$ having a neighbour $y\in L_n \sm K$. Thus $x$ contributes at least 2 to $|\partial K|$, and again  letting $K' = K -x$ we have $\frac{|\partial K'|}{ |K'|}<\frac{|\partial K|}{ |K|}$ because adding the same amount to both the numerator and denumerator to a positive fraction that is smaller than 1 increases the value of the fraction.

The fact that this graph is not hyperbolic can be seen by noticing that every face boundary forms a geodesic triangle ---considering the endpoints and the midpoint of its upper path as the corners of the triangle--- and that these triangles are not thin because there are arbitrarily long face boundaries.

\note{
The fact that this graph is not hyperbolic follows from the following fact
\begin{proposition}
Let \g be a plane hyperbolic non-amenable graph. Then \g has bounded co-degree.
\end{proposition}
\begin{proof}

\end{proof}
}

\medskip

To see that non-amenability is not implied by the other  conditions, let \g be a graph satisfying them all, choose an infinite sequence \seq{v} of vertices of \G, and attach a path of length $n$ to each $v_n$; then the new graph is amenable, but retains all other properties.

It is not clear whether non-amenability or the bounded degree condition is necessary to make \Cr{cor} true, but the following example, due to Itai Benjamini (private communication), shows that we cannot drop both conditions simultaneously: start with the infinite binary tree $T$ embedded in $\R^2$. \Fe\ $n\in \N$, connect the vertices at distance $n$ from a root $o$ with a cycle $C_n$ visiting them all without violating planarity. Then for every $n$ and every edge $e$ of $C_n$, insert $2^{2^n}$ paths of length 2 joining the endpoints of $e$. It is straightforward to check that the resulting graph is hyperbolic. The  paths of length 2  inserted in the last step ensure that this graph has the Liouville property, that is, all bounded harmonic functions are constant, e.g.\ using the methods of \cite{BeKoNon} or \cite{intersection}. It follows from the symmetry of this graph and the Liouville property that its Martin boundary is trivial. The other four boundaries are still homeomorphic to a circle. This example also shows that in \Cr{unbou} we cannot drop the condition that the image of random walk \as\ converges to $\cc$.

This discussion motivates the following problems

\begin{problem} \label{Pram}
Is every planar graph with the Liouville property amenable?
\end{problem}
Note that for Cayley graphs this is true even without the planarity condition \cite{KaVeRan}, while for general graphs it is false even assuming bounded degrees \cite{BeKoNon}.

\begin{problem}\label{Prhyp}
Is there a planar, Gromov-hyperbolic graph with bounded degrees and the Liouville property?
\end{problem}

\acknowledgement{I am grateful to Omer Angel, Ori Gurel-Gurevich and Asaf Nachmias for pointing out some shortcommings in an earlier draft of this paper.}

\bibliographystyle{plain}
\bibliography{../collective}

\begin{thebibliography}{10}

\bibitem{AncNeg}
A.~Ancona.
\newblock {Negatively Curved Manifolds, Elliptic Operators, and the Martin
  Boundary}.
\newblock {\em The Annals of Mathematics}, 125(3):495, May 1987.

\bibitem{Anc}
Alano Ancona.
\newblock {Negatively Curved Manifolds, Elliptic Operators, and the Martin
  Boundary}.
\newblock {\em The Annals of Mathematics}, 125(3):495, May 1987.

\bibitem{AnSchPos}
Michael~T. Anderson and Richard Schoen.
\newblock Positive harmonic functions on complete manifolds of negative
  curvature.
\newblock {\em The Annals of Mathematics}, 121(2):429, March 1985.

\bibitem{AnBaGuNa}
Omer Angel, Martin~T. Barlow, Ori Gurel-Gurevich, and Asaf Nachmias.
\newblock Boundaries of planar graphs, via circle packings.
\newblock {\em {arXiv:1311.3363} [math]}, November 2013.

\bibitem{BC}
J.~R. Baxter and R.~V. Chacon.
\newblock {The equivalence of diffusions on networks to Brownian motion}.
\newblock {\em Contemp.\ Math.}, 26:33--47, 1984.

\bibitem{intersection}
I.~Benjamini, N.~Curien, and A.~Georgakopoulos.
\newblock {The Liouville and the intersection properties are equivalent for
  planar graphs}.
\newblock 17(42):1--5, 2012.

\bibitem{BeSchrHar}
I.~Benjamini and O.~Schramm.
\newblock Harmonic functions on planar and almost planar graphs and manifolds,
  via circle packings.
\newblock {\em Invent.\ math.}, 126:565--587, 1996.

\bibitem{BeSchrST}
I.~Benjamini and O.~Schramm.
\newblock {Random Walks and Harmonic Functions on Infinite Planar Graphs Using
  Square Tilings}.
\newblock {\em Ann.\ Probab.}, 24(3):1219--1238, 1996.

\bibitem{BeKoNon}
Itai Benjamini and Gady Kozma.
\newblock Nonamenable liouville graphs.
\newblock {\em {arXiv:1010.3365} [math]}, October 2010.

\bibitem{BSST}
R.~L. Brooks, C.~A.~B. Smith, A.~H. Stone, and W.~T. Tutte.
\newblock The dissection of rectangles into squares.
\newblock {\em Duke Mathematical Journal}, 7(1):312--340, 1940.

\bibitem{CaFlPaSqu}
J.~W. Cannon, W.~J. Floyd, and W.~R. Parry.
\newblock {Squaring rectangles: The finite Riemann mapping theorem.}
\newblock {Abikoff, William (ed.), The mathematical legacy of Wilhelm Magnus.
  Contemp.\ Math.\ 169}, 1994.

\bibitem{reseaux}
Y.~Colin~de Verdi\`ere, I.~Gitler, and D.~Vertigan.
\newblock {Planar electric networks. II. (Reseaux \'electriques planaires.
  II.)}.
\newblock {\em Comment.\ Math.\ Helv.}, 71(1):144--167, 1996.

\bibitem{Dehn}
M.~Dehn.
\newblock {\"Uber Zerlegung von Rechtecken in Rechtecke.}
\newblock {\em Math. Ann.}, 57:314--332, 1903.

\bibitem{diestelBook05}
Reinhard Diestel.
\newblock {\em Graph {T}heory \emph{(3rd edition)}}.
\newblock Springer-Verlag, 2005.
\newblock \\ Electronic edition available at:\\ {\small\tt
  http://www.math.uni-hamburg.de/home/diestel/books/graph.theory}.

\bibitem{DujSim}
A.J.W. Duijvestijn.
\newblock {Simple perfect squared square of lowest order.}
\newblock {\em J. Comb. Theory, Ser. B}, 25:240--243, 1978.

\bibitem{erschlerICM}
A.~Erschler.
\newblock "poisson-furstenberg boundaries, large-scale geometry and growth of
  groups".
\newblock In {\em Proceedings of the ICM, Hyderabad, India}, pages 681--704,
  2010.

\bibitem{erschler_poissonfurstenberg_2011}
A.~Erschler.
\newblock {Poisson--Furstenberg} boundary of random walks on wreath products
  and free metabelian groups.
\newblock {\em Comment.\ Math.\ Helv.}, 86(1):113--143, 2011.

\bibitem{FurPoi}
H.~Furstenberg.
\newblock {{A Poisson formula for semi-simple Lie groups.}}
\newblock {\em The Annals of Mathematics}, 77(2):335--386, March 1963.

\bibitem{FurRan}
H.~Furstenberg.
\newblock {Random walks and discrete subgroups of {L}ie groups.}
\newblock {\em {Advances Probab. related Topics}}, 1:3--63, 1971.

\bibitem{theta}
A.~Georgakopoulos and V.~Kaimanovich.
\newblock In preparation.

\bibitem{ceff}
A.~Georgakopoulos and V.~Kaimanovich.
\newblock Electrical network reduction with a probabilistic interpretation of
  effective conductance.
\newblock Preprint 2011.

\bibitem{GhHaSur}
E.~Ghys and P.~de~la Harpe.
\newblock {\em Sur les {G}roupes {H}yperboliques d' apr\`es {M}ikhael
  {G}romov}.
\newblock Birkh\"auser, 1990.
\newblock English translation available at
  \small{http://www.umpa.ens-lyon.fr/~ghys/Publis-english.html}.

\bibitem{GurNachRec}
O.~Gurel-Gurevich and A.~Nachmias.
\newblock Recurrence of planar graph limits.
\newblock {The Annals of Mathematics} (to appear).

\bibitem{KaiPoHypAnn}
V.~Kaimanovich.
\newblock The poisson formula for groups with hyperbolic properties.
\newblock {\em The Annals of Mathematics}, 152(3):pp. 659--692, 2000.

\bibitem{kaimanovich_poisson_1996}
V.~A. Kaimanovich and H.~Masur.
\newblock The poisson boundary of the mapping class group.
\newblock {\em Inventiones Mathematicae}, 125(2):221--264, 1996.

\bibitem{KaVeRan}
V.~A. Kaimanovich and A.~M. Vershik.
\newblock Random walks on discrete groups: Boundary and entropy.
\newblock {\em The Annals of Probability}, 11(3):457--490, August 1983.

\bibitem{karatzas}
Ioannis Karatzas and Steven~E. Shreve.
\newblock {\em Brownian Motion and Stochastic Calculus}.
\newblock Springer, 1991.

\bibitem{karlsson_poisson_2006}
A.~Karlsson and W.~Woess.
\newblock The poisson boundary of lamplighter random walks on trees.
\newblock {\em Geometriae Dedicata}, 124(1):95--107, November 2006.

\bibitem{Karr}
A.~F. Karr.
\newblock {\em {Probability.}}
\newblock {Springer Texts in Statistics}, 1993.

\bibitem{kenyon_tilings_1998}
R.~Kenyon.
\newblock Tilings and discrete dirichlet problems.
\newblock {\em Israel Journal of Mathematics}, 105(1):61--84, 1998.

\bibitem{kenyon_dimers_2004}
R.~W. Kenyon and S.~Sheffield.
\newblock Dimers, tilings and trees.
\newblock {\em Journal of Combinatorial Theory, Series B}, 92(2):295--317,
  November 2004.

\bibitem{LyonsBook}
R.~Lyons and Y.~Peres.
\newblock {\em Probability on Trees and Networks}.
\newblock Cambridge University Press.
\newblock In preparation, current version available at {\small\tt
  http://mypage.iu.edu/\string~rdlyons/prbtree/prbtree.html}.

\bibitem{BPY}
J.~W.~Pitman M.~T.~Barlow and M.~Yor.
\newblock {On Walsh's Brownian motions}.
\newblock {\em \em S\'{e}minaire de Probabilit\'{e}s (Strasbourg)},
  23(1):275--293, 1989.

\bibitem{NorCir}
S.~Northshield.
\newblock Circle boundaries of planar graphs.
\newblock {\em Potential Analysis}, 2(4):299--314, December 1993.

\bibitem{ThomassenRichter}
R.B. Richter and C.~Thomassen.
\newblock $3$-connected planar spaces uniquely embed in the sphere.
\newblock {\em Trans.\ Am.\ Math.\ Soc.}, 354:4585--4595, 2002.

\bibitem{RohOded}
S.~Rohde.
\newblock Oded schramm: From circle packing to {SLE}.
\newblock In I.~Benjamini and O.~H\"aggstr\"om, editors, {\em Selected Works of
  Oded Schramm}, Selected Works in Probability and Statistics, pages 3--45.
  Springer New York, 2011.

\bibitem{Short}
H.~{Short~(editor)}.
\newblock Notes on {H}yperbolic {G}roups.
\newblock {\small\tt http://www.cmi.univ-mrs.fr/\string~hamish/}.

\bibitem{thoPla}
C.~Thomassen.
\newblock Planarity and duality of finite and infinite graphs.
\newblock {\em J.~Combin.\ Theory (Series B)}, 29(2):244 -- 271, 1980.

\bibitem{WoessBook09}
Wolfgang Woess.
\newblock {\em {Denumerable Markov chains. Generating functions, boundary
  theory, random walks on trees.}}
\newblock {EMS Textbooks in Mathematics. European Mathematical Society (EMS),
  Z\"urich}, 2009.

\end{thebibliography}

\end{document}